\documentclass[smallextended,numbook,runningheads]{svjour3}   
\usepackage{amsmath}
\smartqed  

\usepackage{amsfonts,amssymb}
\usepackage{epsfig}
\usepackage{GGGraphics}

\journalname{BIT}
\date{ \phantom{b} \vspace{45mm}\phantom{e}}

\def\iu{{\rm i}}
\def\e{{\rm e}}
\def\d{{\rm d}}
\def\Z{{\mathbb Z}}

\def\real{{\mathbb R}}
\def\eps{\varepsilon}
\def\bigo{{\mathcal O}}

\def\signed #1 {{\unskip\nobreak\hfil\penalty50
    \hskip2em\hbox{}\nobreak\hfil\small#1
    \parfillskip=0pt \finalhyphendemerits=0 \par}}

\newcommand{\sfrac}[2]{\mbox{\footnotesize$\displaystyle\frac{#1}{#2}$}}
\newcommand\calM{{\cal M}}
\newcommand\calN{{\cal N}}

\newcommand\calI{{\cal I}}

\newcommand\calK{{\cal K}}
\newcommand\calE{{\cal E}}

\newcommand\calU{{\cal U}}
\newcommand\calR{{\cal R}}

\newcommand\bfd{{\mathbf d}}
\newcommand\bfk{{\mathbf k}}

\newcommand\bfp{{\mathbf p}}
\newcommand\bfq{{\mathbf q}}
\newcommand\bfr{{\mathbf r}}

\newcommand\bfg{{\mathbf g}}
\newcommand\bfv{{\mathbf v}}
\newcommand\bfw{{\mathbf w}}
\newcommand\bfx{{\mathbf x}}

\newcommand\bfz{{\mathbf z}}

\newcommand\bfLambda{{\boldsymbol \Lambda}}
\newcommand\bfvartheta{{\boldsymbol \vartheta}}
\newcommand\bfvarpi{{\boldsymbol \varpi}}
\newcommand\bfomega{{\boldsymbol \omega}}
\newcommand\bfOmega{{\mathbf \Omega}}
\newcommand\bfChi{{\boldsymbol \chi}}
\newcommand\bfPsi{{\mathbf \Psi}}
\newcommand\bfPhi{{\mathbf \Phi}}
\newcommand\bfzero{{\mathbf 0}}
\newcommand\jvec{{\langle j \rangle}}
\newcommand\sinc{\,\hbox{\rm sinc}}
\newcommand\sincws{\hbox{\rm sinc}}
\def\T{{\mathsf T}}

\begin{document}

\title{Long-term analysis of numerical integrators 
for oscillatory Hamiltonian systems under minimal non-resonance conditions
\thanks{This work has been supported by the Fonds National Suisse,
Project No.~200020-144313/1.
}}
\titlerunning{Long-term analysis of oscillatory Hamiltonian systems}

\author{David Cohen \and Ludwig Gauckler \and \\ Ernst Hairer \and Christian Lubich}
\institute{David Cohen \at Matematik och matematisk statistik,
Ume\r{a} universitet, SE-90187 Ume\r{a}, Sweden. \\
\email{david.cohen@math.umu.se}
\and
Ludwig Gauckler \at Institut f\"ur Mathematik, TU Berlin, Stra\ss e des 17.\ Juni 136, D-10623 Berlin, Germany. \\
\email{gauckler@math.tu-berlin.de}
\and
Ernst Hairer \at Section de math\'ematiques, 2-4 rue du Li\`evre,
Universit\'e de Gen\`eve, CH-1211
Gen\`eve 4,
Switzerland.\\
\email{Ernst.Hairer@unige.ch} 
\and
Christian Lubich \at Mathematisches Institut, Universit\"at T\"ubingen, Auf der Morgenstelle,
  D-72076 T\"ubingen, Germany.\\
  \email{Lubich@na.uni-tuebingen.de}}

\authorrunning{D. Cohen, L. Gauckler, E. Hairer and Ch. Lubich}

\date{}

\maketitle
\begin{abstract} For trigonometric and modified trigonometric integrators applied to oscillatory Hamiltonian differential equations with one or several constant high frequencies, near-conservation of the total and oscillatory energies are shown over time scales that cover arbitrary negative powers of the step size. 
This requires non-resonance conditions between the step size and the frequencies, but in
contrast to previous results the results do not require any non-resonance conditions among the frequencies. The proof uses modulated Fourier expansions with appropriately modified frequencies. 
\end{abstract}

\section{Introduction}
\label{sect:intro}
This paper is concerned with the energy behaviour over long times of numerical methods for oscillatory differential equations with one or several constant high frequencies:
\begin{equation}\label{odec}
\ddot q_j + \omega_j^2 q_j =
 -\nabla_j U(\bfq ) , \qquad j=0,\dots,\ell ,
\end{equation}
where $\bfq = (q_{0},q_{1},\ldots ,q_{\ell})$ and the frequencies are $\omega_0=0$ and
\begin{equation}\label{frequencies}
\omega_{j} \ge \frac1{\eps}, \quad
0<\eps\ll 1 , \qquad j=1,\dots,\ell.
\end{equation}
The coupling potential $U$ is smooth with derivatives bounded independently of~$\eps$.

The numerical long-time near-conservation of energy for such equations has already been studied before for various numerical integrators, which will be considered also here: for trigonometric integrators in \cite{hairer00lec,cohen05nec} and  \cite[Chapter XIII]{hairer06gni}, for the St\"ormer--Verlet method in  \cite{hairer00ecb}, for an implicit-explicit (IMEX) method in \cite{stern08iev,mclachlan13mti}. Most of these results concern the single-frequency case ($\ell=1$). A nontrivial extension to the multi-frequency case ($\ell>1$) is given
in~\cite{cohen05nec}. 

The proofs in those papers 
require numerical non-resonance conditions: first, it is needed that the product $h\omega_j$ of the step size $h$ with the high frequencies (or $h\widetilde\omega_j$ for methods that effectively work with different numerical frequencies $\widetilde\omega_j$, such as the St\"ormer--Verlet and IMEX methods) is bounded away from integral multiples of $\pi$ by a distance substantially larger than $h$, e.g., by $\sqrt{h}$. Numerical experiments, e.g., in \cite[Chapter XIII]{hairer06gni}, show clearly that this numerical non-resonance condition is  necessary for a satisfactory numerical energy behaviour.

In this paper we show that without any further non-resonance condition, the {\it slow energy}, i.e., the total energy minus the oscillatory energy, remains well conserved over long times $t\le h^{-N}$ for an arbitrary integer $N$, in the numerically interesting range $h\ge c_0\eps$,
provided the total energy remains
bounded over such an interval.

To obtain also near-conservation of the {\it total and oscillatory energies} over times $t\le h^{-N}$, it is required in \cite{cohen05nec} and \cite[Chapter XIII]{hairer06gni} that sums of $\pm h\omega_j$ with at most $N+1$ terms
must stay away from integral multiples of $2\pi$. Here we will show that it suffices that they are away from {\it non-zero} integral multiples of $2\pi$, and we present numerical results that  illustrate the necessity of such a non-resonance condition between the step size and the frequencies.

 Away from numerical near-resonances between the step size and the frequencies, our results will be uniform in the frequencies, without imposing any non-resonance condition among the frequencies.

What enables us to obtain long-time near-conservation of the total, slow  and oscillatory energies under much less restrictive conditions than in the existing literature, is that we adopt ideas and techniques from \cite{gauckler13esi} for the exact solution and combine them with those of \cite{hairer00lec} and \cite{cohen05nec} for the numerical solution. In \cite{gauckler13esi}  and more recently also in \cite{bambusi13nfa} it is shown that the oscillatory energy, i.e., the sum of the harmonic energies $\frac12 |\dot q_j|^2 +
\frac12 \omega_j^2 |q_j|^2$ over $j=1,\dots,\ell$, is nearly conserved over times $\eps^{-N}$ for arbitrary integers $N$, with estimates that are uniform in the frequencies $\omega_j$ satisfying (\ref{frequencies}). In the proof of this result, integral linear combinations of the frequencies are regrouped into non-resonant and near-resonant ones, with a gap appearing between them whose size depends only on $\ell$ and $N$, and then the frequencies are modified such that the near-resonances become exact resonances. With these modified frequencies, one then uses a modulated Fourier expansion and its almost-invariants  as done previously in the literature. We will proceed in a similar way in this paper for the numerical problem and prove a numerical counterpart to the analytical result of \cite{gauckler13esi} and \cite{bambusi13nfa}.

After the preparatory Section~\ref{sect:prep} we state, in Section~\ref{sect:results}, our main results on the long-time near-conservation of the slow and oscillatory energies along numerical solutions of (\ref{odec}).
After some illustrative numerical experiments in Section~\ref{sect:numexp},
the theoretical results are proved in the remaining sections.
We introduce appropriately modified frequencies in Section~\ref{sect:gap} and use them  in the modulated Fourier expansion constructed in Section~\ref{sect:mfe}. The result on bounds and approximation properties of the modulated Fourier expansion, which is stated in Section~\ref{sect:mfe}, is proved in Section~\ref{sect:proof-mfe}. Two almost-invariants of the modulation system, which are close to the slow and oscillatory energies, are studied in Section~\ref{sect:invariants}. We are then in the position to prove the main results in Section~\ref{sect:proof-thm12}.

\section{Preparation}
\label{sect:prep}

\subsection{Oscillatory, slow and total energies and equations of motion}
For momenta $\bfp = (p_{0},p_{1},\ldots ,p_{\ell})$ and positions
$\bfq = (q_{0},q_{1},\ldots ,q_{\ell})$ with $p_{j},q_{j}\in\real^{d_{j}}$
we consider the  Hamiltonian 
\[
H(\bfp ,\bfq ) =  H_{\bfomega}(\bfp ,\bfq ) + H_{\rm slow} (\bfp ,\bfq ) ,
\]
where the oscillatory and slow-motion energies are given by
\[
H_{\bfomega}(\bfp ,\bfq )  =  \sum_{j=1}^\ell \sfrac 12\Bigl( |p_{j}|^2 +
 \omega_{j}^2 \,|q_{j}|^2\Bigr), \qquad
H_{\rm slow}(\bfp ,\bfq ) = \sfrac 12  |p_{0}|^2 + U(\bfq ) .
\]
We assume high frequencies satisfying (\ref{frequencies}). 
If the oscillatory energy is bounded by a constant independent of $\eps$, then
$q_j=\bigo(\eps)$ for $j=1,\dots,\ell$, and we have that the slow energy $H_{\rm slow}(\bfp ,\bfq )$ is $\bigo(\eps)$ close to the energy of the isolated slow system $H_0(p_0,q_0) = \frac 12  |p_{0}|^2 + U(q_0,0,\dots,0)$.

The equations of motion are (\ref{odec}) or, in vector notation,
\begin{equation}\label{ode}
\ddot \bfq + \bfOmega^2 \bfq = \bfg (\bfq) 
\end{equation}
with the nonlinearity 
$\bfg (\bfq)= -\nabla U(\bfq ) $, and where $\bfOmega$ is the
diagonal matrix with entries $\omega_j$. 

\subsection{Trigonometric integrators}
For the numerical solution of (\ref{ode}) we consider trigonometric
methods as studied in \cite[Chapter XIII]{hairer06gni}. 
With the step size $h$, they are given in two-step form by
\begin{equation}\label{trigo}
\begin{array}{rcl}
\bfq_{n+1} - 2 \cos (h\bfOmega ) \bfq_n + \bfq_{n-1} &=& h^2 \bfPsi\bfg (\bfPhi \bfq_n )\\[2mm]
2h \sinc (h\bfOmega ) \bfp_n &=& \bfq_{n+1} - \bfq_{n-1} ,
\end{array}
\end{equation}
where $\bfPsi = \psi (h\bfOmega )$ and $\bfPhi = \phi (h\bfOmega )$ with real-valued
bounded functions $\psi$ and $\phi$ satisfying $\psi (0)=\phi (0)=1$,
and $\sincws (\xi ) = \sin (\xi )/\xi$. For starting the computation we put
$\bfq_0 = \bfq (0)$, $\bfp_0 = \bfp (0)$ and we compute the
approximation $\bfq_1$ by putting $n=0$ in (\ref{trigo}) and by
eliminating $\bfq_{-1}$. This yields
\begin{equation}\label{trigo-q1}
\bfq_1 = \cos (h\bfOmega ) \bfq_0 + h \sinc (h\bfOmega ) \bfp_0 +
\sfrac 12 h^2 \bfPsi \bfg(\bfPhi \bfq_0 ) .
\end{equation}
It is known from \cite[Section XIII.2.2]{hairer06gni} that the method is symplectic if and only if 
\begin{equation}\label{symplectic}
\psi (h\omega_j) = \sincws (h\omega_j ) \phi (h\omega_j )\qquad
\hbox{for}\quad j=1,\ldots ,\ell.
\end{equation}

\subsection{Modified trigonometric integrators}\label{sect:mod-trigo}
We further consider methods defined by
\begin{equation}\label{trigo-mod}
\begin{array}{rcl}
\bfq_{n+1} - 2 \cos (h\widetilde\bfOmega ) \bfq_n + \bfq_{n-1} &=& 
h^2 \widetilde\bfPsi\bfg (\widetilde\bfPhi \bfq_n )\\[2mm]
2h\, \widetilde \bfChi\, \bfp_n &=& \bfq_{n+1} - \bfq_{n-1} ,
\end{array}
\end{equation}
where $\widetilde\bfOmega$ is a diagonal matrix with entries $\widetilde\omega_j$
such that $\widetilde\omega_0 =0$,
$\widetilde\bfPsi = \psi (h\widetilde\bfOmega )$,
$\widetilde\bfPhi = \phi (h\widetilde\bfOmega )$, and
$\widetilde\bfChi = \chi (h\widetilde\bfOmega )$ with $\chi (0)=1$.
The choice
$\widetilde\bfChi = \bfOmega^{-1} \widetilde\bfOmega\sinc (h\widetilde\bfOmega ) $
is proposed in \cite{mclachlan13mti}. The method is symplectic if and only if 
\[
\psi (h\widetilde\omega_j) = \chi (h\widetilde\omega_j ) \phi (h\widetilde\omega_j )\qquad
\hbox{for}\quad j=1,\ldots ,\ell.
\]

An important class of such symplectic methods is given by
\begin{equation}\label{trigo-alpha}
\begin{array}{rcl}
\bfq_{n+1} - 2 \bfq_n + \bfq_{n-1} &+& h^2\bfOmega^2 \bfq_n +
\alpha h^2 \bfOmega^2 (\bfq_{n+1} - 2 \bfq_n + \bfq_{n-1} ) =
 h^2 \bfg (\bfq_n ) \\[2mm]
2h\, \bfp_n &=& (I + \alpha h^2 \bfOmega^2 )(\bfq_{n+1} - \bfq_{n-1} ).
\end{array}
\end{equation}
This can be written as a method (\ref{trigo-mod}) by defining $h\widetilde\omega_j
\in [0,\pi ]$
through (see \cite{hairer00ecb}; we omit the subscript $j$)
\[
\cos (h\widetilde\omega ) = \frac{1+(\alpha -\frac 12) h^2\omega^2}{1+\alpha h^2\omega^2}
\quad\hbox{or equivalently}\quad
\sin \bigl( \tfrac 12 h\widetilde\omega \bigr) = \frac {\tfrac 12 h\omega}
{\sqrt{1+\alpha h^2 \omega^2 }} ,
\]
provided that $h\omega < 2/\sqrt{1-4\alpha}$ if $\alpha < 1/4$, and without any
restriction on $h\omega$ if $\alpha \ge 1/4$. With these modified
frequencies, the method (\ref{trigo-alpha}) becomes (\ref{trigo-mod}) with
\[
\phi (\xi )=1,\qquad \psi (\xi ) = \chi (\xi )= 1 -4\alpha \sin^2 \bigl(\tfrac 12 \xi \bigr).
\]
The St\"ormer--Verlet method is the special case $\alpha =0$ of (\ref{trigo-alpha}),
and the implicit-explicit (or IMEX) integrator of \cite{zhang97cis} and \cite{stern08iev}
is the special case $\alpha = 1/4$.

\section{Main results on energy conservation}
\label{sect:results}

We prove results on numerical energy conservation for trigonometric and modified trigonometric
integrators. 
The technique of proof is related to that of
\cite{gauckler13esi} where a gap condition is created by suitably
modifying the frequencies. 

We collect assumptions that are relevant for all theorems to be
presented in this work.
\medskip

\noindent
\bf Assumption A. \rm
In addition to (\ref{frequencies}) we assume the following:
\begin{itemize}
\item
The total energy
of the initial values is bounded independently of
$\varepsilon$,
\begin{equation}\label{bound-energy}
 H\bigl(\bfp (0) , \bfq (0) \bigr)  \le E .
\end{equation}
\item
There is a radius $\rho >0$ and a set $K\subset \real^{d_0}$ such that the potential
$U(\bfq )$ is bounded and has bounded derivatives of all orders in a $\rho$-neighbourhood of
$K\times 0 \times \cdots\times 0$. We denote this $\rho$-neighbourhood by $K_\rho$.
\item
The numerical solution values $\bfPhi \bfq_n$
(or $\widetilde \bfPhi \bfq_n$) stay in $K_{\rho/2}$. 
\item The step size $h$ satisfies
 $h/\varepsilon \ge c_0 >0$.
\item
The frequencies $\omega_j$ are such that
\begin{equation}\label{eq-nonres}
| \sin (h\omega_j ) | \ge \kappa=\kappa(h) \ge \sqrt h \qquad\hbox{for}\quad   j=1,\dots,\ell .
\end{equation}
\end{itemize}
For modified trigonometric integrators this is assumed for the frequencies $\widetilde\omega_j$ instead of $\omega_j$.

\medskip

\noindent
\bf Assumption B. \rm
The filter functions $\phi$ and $\psi$ of the method (\ref{trigo}) 
are such that the function
\begin{equation}\label{sigma}
\sigma (\xi ) = \sincws (\xi )\, \frac{\phi (\xi )}{ \psi (\xi )} 
\end{equation}
is bounded from below and above:
\[
0 < c_1 \le \sigma (h\omega_j) \le  C_1 \qquad\hbox{for}\quad j=1,\ldots , \ell,
\]
or the same estimate holds for $-\sigma$ instead of $\sigma$.

\begin{remark}
The first three items of Assumption A are necessary for the following results.
The last two items can be relaxed to the conditions $h/\eps \ge c_0 h^\beta$
and $\kappa (h) \ge h^\gamma$ with positive $\beta$ and $\gamma$ smaller
than $1$, which
lead to weaker estimates. 

Assumption B is needed in our proofs. Numerical experiments indicate, however,
that different signs among the $\sigma (h\omega_j )$ do not necessarily lead
to a different energy behaviour.
\end{remark}

\subsection{Energy conservation for trigonometric integrators}
\label{sect:multiple}

As in \cite{cohen05nec}
we consider the modified oscillatory energy, with $\sigma (\xi )$ from (\ref{sigma}),
\begin{equation}\label{mod-energies}
H_{\bfomega}^*(\bfp ,\bfq )  = 
  \sum_{j=1}^\ell \sigma (h\omega_j) \sfrac 12 \Bigl( |p_{j}|^2 + \omega_{j}^2 \,|q_{j}|^2\Bigr) .
\end{equation}
If $\sigma (h\omega_j )=1$ for $j=1,\ldots ,\ell$, this  expression is identical
to the oscillatory energy~$H_\bfomega$. This condition on $\sigma$ is equivalent to the symplecticity of the numerical flow
defined by (\ref{trigo}), see (\ref{symplectic}).

\begin{theorem}\label{thm:main}
We fix an arbitrary integer $N\ge 1$ and $0< \delta \le 1/4$. Then there exists $h_0>0$
such that under Assumptions A and B, the numerical solution obtained by method
(\ref{trigo}) satisfies, for $h\le h_0$,
\[
H_{\rm slow}(\bfp_n , \bfq_n ) = H_{\rm slow}(\bfp_0 , \bfq_0 ) + \bigo (h^{1-\delta}) 
\quad \hbox{for } \quad 0 \le nh \le h^{-N} ,
\]
as long as $H(\bfp_n,\bfq_n) \le \hbox{Const.}$
If, in addition, the step size and the frequencies satisfy the following numerical non-resonance condition:
\begin{equation} \label{num-nonres}
\begin{array}{l}
\hbox{Sums of $\pm h\omega_j$ with at most $N+1$ terms are bounded away from} 
\\
\hbox{nonzero integral multiples of $2\pi$ with a distance of at least $ \sqrt h$,}
\end{array}
\end{equation}
then we further have, with $\kappa$ from (\ref{eq-nonres}),
\[
H_\bfomega^*(\bfp_n , \bfq_n ) = H_\bfomega^*(\bfp_0 , \bfq_0 ) + 
\bigo (h^{1-\delta}/\kappa )
\quad \hbox{for } \quad 0 \le nh \le h^{-N} . 
\]
The constants symbolized by $\bigo$ are independent of $n$, $h$, $\varepsilon$,
and $\omega_j$, but depend on $\ell$, $N$, $\delta$, and the constants in 
Assumptions A and~B.
The maximal step size
$h_0$ is independent of the frequencies $\omega_j$.
\end{theorem}

For vectors $\bfk=(k_1,\dots,k_\ell )\in\Z^\ell$ of integers and the vector $\bfomega = (\omega_1,\dots,\omega_\ell )$ of frequencies we write
\[
\bfk\cdot\bfomega = \sum_{j=1}^\ell k_j\omega_j \qquad\hbox{and}\qquad 
\|\bfk\| = \sum_{j=1}^\ell |k_j|.
\]
Condition (\ref{num-nonres}) can then be rewritten as
\begin{equation}\label{num-nonres-noamal}
|h (\bfk\cdot\bfomega) - r\,2\pi| \ge \sqrt{h} \quad\  \hbox{ for all } r \in \Z, \; r \ne 0, \hbox{ for all } \bfk\in \Z^\ell \hbox{ with } \|\bfk\| \le N+1.
\end{equation}
Theorem \ref{thm:main} is related to the results of \cite{cohen05nec}. A substantial difference
is that here we do not require the numerical non-resonance condition from
\cite{cohen05nec}, which reads
\begin{equation}\label{num-nonres-altpapier}
\Bigl| \sin \Bigl( \sfrac h2 \, \bfk\cdot \bfomega \Bigr) \Bigr| \ge
 \sqrt h \qquad \hbox{for all
$\bfk\in \Z^\ell \setminus \calM$ with $\| \bfk \| \le N+1$}
\end{equation}
with the resonance module $\calM = \{\, \bfk\in \Z^\ell \, : \, \bfk\cdot \bfomega = 0 \,\}$.
This condition is more restrictive than the numerical non-resonance condition of Assumption A and of (\ref{num-nonres}), in particular in that the case $r=0$ is not required in (\ref{num-nonres-noamal}). In contrast to \cite{cohen05nec}, where near-resonances among the frequencies are excluded,
Theorem~\ref{thm:main} is uniform in the choice of the frequencies
$\omega_1 ,\ldots , \omega_\ell$. On the other hand, 
under the non-resonance condition (\ref{num-nonres-altpapier}) and under further 
assumptions on the filter functions, 
the article \cite{cohen05nec} gives improved near-conservation estimates. 

We note the following direct 
corollary on the conservation of the total and
oscillatory energies.

\begin{corollary} \label{cor:H}
If, in addition to the assumptions of Theorem~\ref{thm:main} including (\ref{num-nonres}), the method (\ref{trigo})
is symplectic, then
\[
\begin{array}{rcl}
H(\bfp_n , \bfq_n ) &=& H(\bfp_0 , \bfq_0 ) + \bigo (h^{1-\delta}/\kappa )\\[2mm]
H_\bfomega (\bfp_n , \bfq_n ) &=& H_\bfomega (\bfp_0 , \bfq_0 ) + 
\bigo (h^{1-\delta}/\kappa )
\end{array}
\quad \hbox{for } \quad 0 \le nh \le h^{-N} . 
\]
\end{corollary}

\begin{remark}
If $\sigma(h\omega_j)>0$ for $j=1,\dots,\ell$, the 
modified oscillatory energy $H_\bfomega^*$ 
is the oscillatory energy $H_\bfomega$ in transformed variables 
$\widehat{q}_j = \sigma(h\omega_j)^{1/2} \, q_j$ 
and $\widehat{p}_j = \sigma(h\omega_j)^{1/2} \, p_j$. 
The method (\ref{trigo}) in these variables is still 
of the form (\ref{trigo}) with modified filter functions 
$\widehat{\psi}(h\omega_j) = \sigma(h\omega_j)^{1/2} \, \psi(h\omega_j)$ 
and $\widehat{\phi}(h\omega_j) = \sigma(h\omega_j)^{-1/2} \, \phi(h\omega_j)$. 
The method is thus symplectic in the transformed variables. 
This indicates
why the modified oscillatory energy instead 
of the oscillatory energy shows up in Theorem~\ref{thm:main}.
\end{remark}

\subsection{Energy conservation for modified trigonometric integrators}

We have the following result for the symplectic class of methods (\ref{trigo-alpha}),
which contains the St\"ormer--Verlet scheme and the IMEX integrator
of \cite{zhang97cis} and \cite{stern08iev} as special cases. Here we introduce
\[
H_{\widetilde\bfomega}^*(\bfp ,\bfq )  = 
  \sum_{j=1}^\ell \sigma (h\widetilde\omega_j) \sfrac 12 \Bigl( 
 \Bigl( \frac {\chi(h\widetilde\omega_j)} {\sincws(h\widetilde\omega_j)}\Bigr)^2 |p_{j}|^2 + \widetilde \omega_{j}^2 \,|q_{j}|^2\Bigr) .
\]

For this class of methods, the non-resonance condition (\ref{eq-nonres}) as well as the condition on $\sigma$ of Assumption~B are satisfied under a step size restriction.

\begin{theorem}\label{thm:main-alpha}
We consider the method (\ref{trigo-alpha}). If $\alpha \ge 1/4$, we assume that the step size
is restricted by $h\omega_j \le \text{const}$ for $j=1,\ldots ,\ell$, 
and if $\alpha<1/4$, we assume that $h\omega_j\le 2\theta/\sqrt{1-4\alpha}$ 
for $j=1,\dots,\ell$ with $\theta<1$.
We fix an arbitrary integer $N\ge 1$ and $0< \delta \le 1/4$. Then there exists
$h_0 >0$ such that,
under Assumption A without condition (\ref{eq-nonres}),
the numerical solution obtained by method
(\ref{trigo-alpha}) satisfies, for $h\le h_0$ under the above step size restriction,
\[
H_{\rm slow}(\bfp_n , \bfq_n ) = H_{\rm slow}(\bfp_0 , \bfq_0 ) + \bigo (h^{1-\delta}) 
\quad \hbox{for } \quad 0 \le nh \le h^{-N} ,
\]
as long as $H(\bfp_n,\bfq_n) \le \hbox{Const.}$
If, in addition, the step size $h$ and the frequencies $\widetilde\omega_j$ satisfy the numerical non-resonance condition (\ref{num-nonres}),
then we further have, with $\kappa$ from (\ref{eq-nonres}),
\[
H_{\widetilde\bfomega}^*(\bfp_n , \bfq_n ) = H_{\widetilde\bfomega}^*(\bfp_0 , \bfq_0 ) + 
\bigo (h^{1-\delta}/\kappa )
\quad \hbox{for } \quad 0 \le nh \le h^{-N} . 
\]
The constants symbolized by $\bigo$ are independent of $n$, $h$, $\varepsilon$,
and $\omega_j$, but depend on $\ell$, $N$, $\delta$, $\theta$ and the constants in 
Assumption A. The threshold
$h_0$ is independent of the frequencies $\omega_j$.
\end{theorem}

Theorem~\ref{thm:main-alpha} follows from Theorem~\ref{thm:main},
using the frequencies $\widetilde\omega_j$ instead of $\omega_j$ and the transformed momenta $\widetilde p_j= \bigl(\chi (h\widetilde\omega_j )/\sinc(h\widetilde\omega_j ) \bigr) p_j$: Since $h\widetilde\omega_j \in [0,\pi]$, condition (\ref{eq-nonres}) of Assumption~A is satisfied with some $\kappa$ independent of $h$ under the step size restriction of Theorem~\ref{thm:main-alpha}. Concerning Assumption B, the condition on $\sigma$ is satisfied under the step size restriction of Theorem~\ref{thm:main-alpha}. This also implies that $|\chi(\xi)/\sincws(\xi)| = |\psi(\xi)/\sincws(\xi)| = |\phi(\xi)/\sigma(\xi)|$ is bounded for $\xi=h\widetilde\omega_j$, $j=1,\dots,\ell$. The step size restriction thus further ensures that the bounded energy condition~(\ref{bound-energy}) of Assumption~A is also satisfied in the transformed variables and with the numerical frequencies $\widetilde{\omega}_j$ instead of $\omega_j$.

\begin{remark}
A result like Corollary \ref{cor:H} is not valid for the class of symplectic methods
(\ref{trigo-alpha}), because $H_\bfomega$ and $H_{\widetilde\bfomega}^*$ differ by
terms of size $\bigo (1)$; see also \cite{hairer00ecb} and \cite[Section XIII.8]{hairer06gni} for the St\"ormer--Verlet method. 
For the IMEX method ($\alpha = 1/4$) one obtains 
\[
H_{\widetilde\bfomega}^*(\bfp ,\bfq )  = 
  \sum_{j=1}^\ell \frac{\widetilde\omega_j}{\omega_j} \sfrac 12
 \Bigl(  |p_{j}|^2 +  \omega_{j}^2 \,|q_{j}|^2\Bigr) .
\]
Therefore, Corollary~\ref{cor:H} is valid for the IMEX scheme for the
special case of a single high frequency.
\end{remark}

\section{Numerical experiments under numerical resonances}
\label{sect:numexp}

The following experiments illustrate the numerical energy behaviour in
situations, for which condition (\ref{num-nonres}) in Theorem
\ref{thm:main} is only fulfilled for small values of $N$,
so that the time intervals $0\le nh\le h^{-N}$ in Theorems~\ref{thm:main}
and \ref{thm:main-alpha} are {\it not} very long.
It will be observed
that {\it numerical resonances} then play
a significant role in the preservation of the total oscillatory energy
on longer time intervals.

Unless otherwise stated
we use the symplectic trigonometric integrator (\ref{trigo}) with filter functions
\cite{deuflhard79aso}
\begin{equation}\label{deuflhard}
\phi (\xi ) =1 , \qquad \psi (\xi ) = \sinc (\xi ).
\end{equation}

\bigskip\it\noindent
Experiment 1. Problem with one degree of freedom. \rm
We consider the scalar differential equation
\[
\ddot q + \omega^2 q = -\nabla U(q), \qquad U(q) = q^3 + q^4
\]
with initial values $q(0)=0.1\,\omega^{-1}$ and $\dot q (0)=1$.
Figure~\ref{fig:simpleprob} shows the deviation of the
oscillatory energy as a function of time on the interval $[0,t_{end}]$.
The two pictures
correspond to the cases $h\omega = 2\pi /3$ and $h\omega =\pi /2$,
each one for seven different
values of $\omega = \eps^{-1}$ (large deviations correspond to smaller values of~$\omega$).
The different frequencies are chosen so that for fixed time $t$ the difference of
two consecutive deviations is nearly constant. This allows us to guess that
the dominant term in the deviation behaves like $\bigo ( t \varepsilon^2 )$
for $h\omega = 2\pi/3$, and like $\bigo (t\varepsilon^3 )$ for $h\omega = \pi /2$.
We note that condition (\ref{num-nonres}) is satisfied with $N=1$ for $h\omega =2\pi/3$ and
with $N=2$ for $h\omega = \pi/2$. The near-preservation of the oscillatory energy
on intervals of length $\bigo (h^{-N})$, stated in Corollary 3.1, can be observed
in the numerical experiment.

\begin{figure}[t]
 \begin{picture}(0,0)
  \epsfig{file=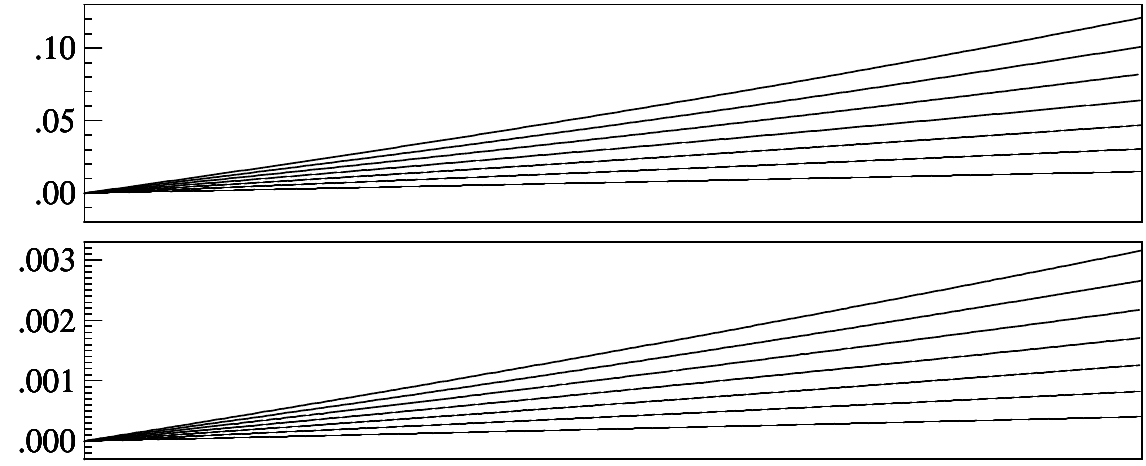}
 \end{picture}%
\begin{picture}(331.4,133.6)( -24.2, 462.4)
  \GGGput[    0,  180]( 98.77,594.66){$t_{end}=200$}
  \GGGput[    0,  150]( 31.80,594.66){$ h\omega = 2\pi/3,$}
  \GGGput[    0,  180]( 98.77,526.37){$t_{end}=2\,000$}
  \GGGput[    0,  150]( 31.80,526.37){$ h\omega = \pi/2,$}
 \end{picture}
\caption{Problem with one degree of freedom:
deviation of the numerical oscillatory energy as a function of time $t$.
In the upper picture the curves correspond to $\omega = 100\cdot \ell^{-1/2}$, in the
lower picture to $\omega = 100\cdot \ell^{-1/3}$, in each case for
$\ell = 1,2,\ldots, 7$.}\label{fig:simpleprob}
\end{figure}

\bigskip\it\noindent
Experiment 2. Alternating stiff and soft springs. \rm
We consider the motion of alternating stiff harmonic and soft nonlinear springs as 
discussed in \cite[Section~I.5 and Chapter XIII]{hairer06gni}. The corresponding
differential equation is of the form treated in this article with only one
high frequency~$\omega= \varepsilon^{-1}$. Along the exact solution of the problem,
the oscillatory energy satisfies $H_\bfomega (\bfp ,\bfq ) = H_\bfomega (\bfp_0 ,\bfq_0 )
+ \bigo (\varepsilon )$ on exponentially long time intervals.
We apply various trigonometric integrators to this problem, and we are mainly
interested in using step sizes for which $h\omega \approx 2\pi r/k$ with integer
values for~$r$ and~$k$. 

Figure~\ref{fig:fpu_ho_ttt} shows the deviation of the
oscillatory energy for method (\ref{deuflhard}), applied to the problem
with $\omega =50$ and step size according to $h\omega = 2\pi/3$.
We observe that the deviation
of the oscillatory energy is of size $\bigo (\varepsilon )$ on an interval of length
$\bigo (\varepsilon^{-2})$. On longer time intervals the deviation behaves like a
random walk. This is illustrated by computing
trajectories with slightly perturbed initial values. Repeating the experiment
with other values of $\omega$, one finds that the deviation of the numerical
oscillatory energy
behaves like $\bigo (\varepsilon ) + \bigo (\varepsilon^2 \sqrt t )$.
Such a random walk behaviour has already been observed with 
computations by the simplified Takahashi--Imada method
\cite{hairer09oec}.
A similar experiment with $h\omega = \pi /2$ leads to a
$\bigo (\varepsilon ) + \bigo (\varepsilon^3 \sqrt t )$ behaviour. 
Figure~\ref{fig:fpu_ho_loop} shows the maximum deviation of the
numerical oscillatory energy until a fixed time $t=100\,000$ as a function
of $h\omega$.

\begin{figure}[t]
 \begin{picture}(0,0)
  \epsfig{file=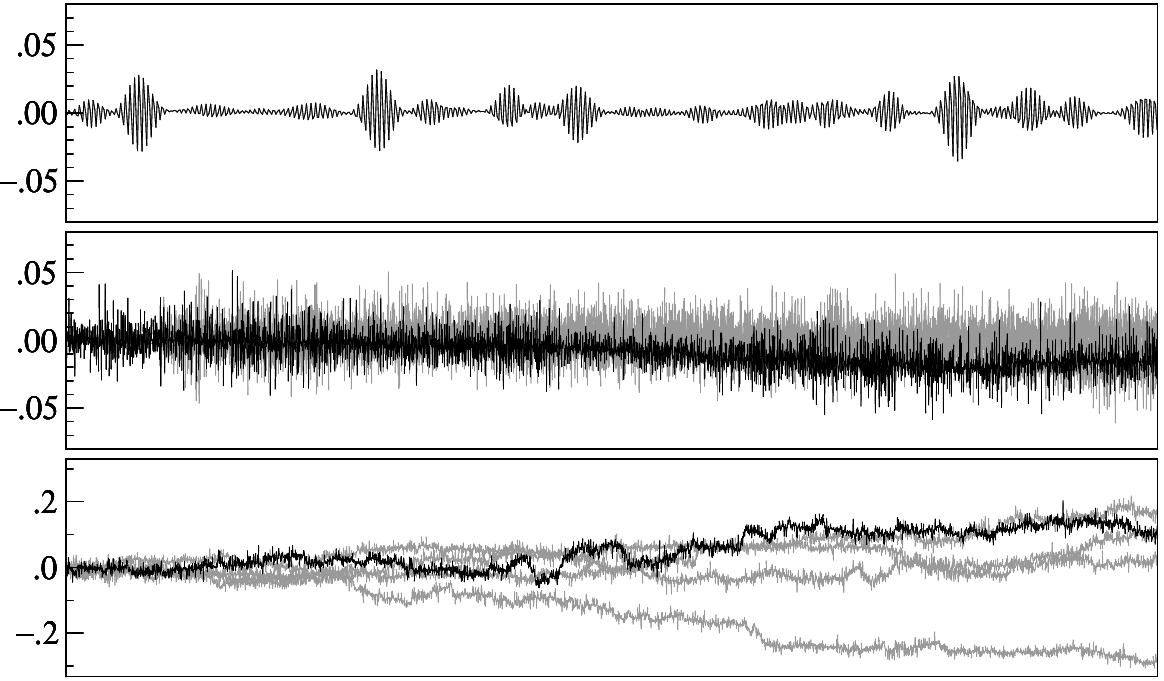}
 \end{picture}%
\begin{picture}(335.6,196.2)( -19.9, 204.9)
  \GGGput[    0,  -60]( 31.44,337.14){$t_{end}=30$}
  \GGGput[    0,  -60]( 31.44,271.74){$t_{end}=3\,000$}
  \GGGput[    0,  -60]( 31.44,206.21){$t_{end}=300\,000$}
 \end{picture}
\caption{Alternating stiff and soft springs:
deviation of the numerical oscillatory energy as a function of time $t$. We consider
the method (\ref{deuflhard}) with
$\omega =50$, $h\omega = 2\pi /3$, and initial values from \cite{hairer06gni}
(black curve). The gray curves are obtained when one of the
initial values is perturbed to $1 \pm \delta$ with $|\delta | \le 7\cdot 10^{-16}$.}\label{fig:fpu_ho_ttt}
\end{figure}

For the modified trigonometric integrators (\ref{trigo-mod}) we expect similar
results, because they can be interpreted as trigonometric integrators
with modified frequencies. We apply the IMEX integrator (\ref{trigo-alpha})
with $\alpha =0.25$ to the problem with alternating stiff and soft springs,
where we take $\omega = 25$. In Figure~\ref{fig:fpu_ho_imex} we
present the results for two different step sizes:
$h\omega = \sqrt{12}$ and
$h\omega = 2$ which, by the relation of Section~\ref{sect:mod-trigo},
correspond to $h\widetilde\omega = 2\pi /3$ and $h\widetilde\omega = \pi /2$,
respectively. Similar as for the trigonometric method we observe
a random walk behaviour: $\bigo (\varepsilon ) + \bigo (\varepsilon^2 \sqrt t )$
for $h\omega = \sqrt{12}$, and
$\bigo (\varepsilon ) + \bigo (\varepsilon^3 \sqrt t )$
for $h\omega = 2$. The additional factor $\varepsilon$ in the second
experiment can be guessed from the figures, because a similar
quantitative behaviour is observed on an interval that is $\omega^2$ times longer.

\begin{figure}[t]
 \begin{picture}(0,0)
  \epsfig{file=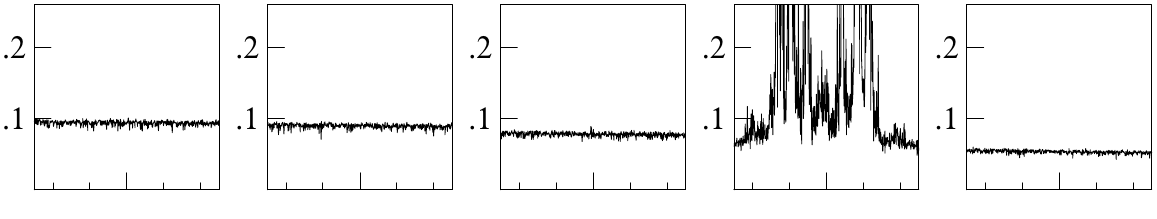}
 \end{picture}%
\begin{picture}(334.5, 63.0)(  -4.5, 259.0)
  \GGGput[   50,  115]( 33.49,267.40){$\pi/4$}
  \GGGput[   50,  115](100.70,267.40){$\pi/3$}
  \GGGput[   50,  115](167.79,267.40){$\pi/2$}
  \GGGput[   50,  115](235.00,267.40){$2\pi/3$}
  \GGGput[   50,  115](302.09,267.40){$3\pi/4$}
 \end{picture}
\caption{Alternating stiff and soft springs:
maximum deviation of the numerical
oscillatory energy on a time interval of length $100\,000$
as a function of $h\omega$. The value of $\omega$ is fixed. Each picture shows this deviation
on an equidistant grid ($591$ points) of an $h\omega$-interval of length~$0.1$.}
\label{fig:fpu_ho_loop}
\end{figure}

\begin{figure}[t]
 \begin{picture}(0,0)
  \epsfig{file=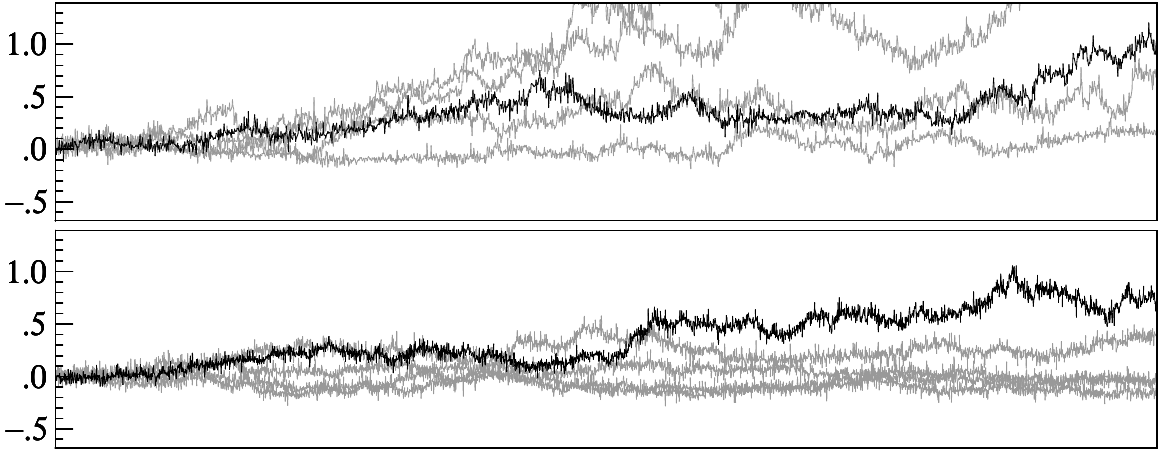}
 \end{picture}%
\begin{picture}(335.5,130.7)( -17.0, 466.7)
  \GGGput[    0,  -60]( 31.68,533.47){$t_{end}=300\,000$}
  \GGGput[    0,  140]( 31.68,596.11){$\omega = 25,~~ h\omega = \sqrt{12}$}
  \GGGput[    0,  -60]( 31.68,468.07){$t_{end}=300\,000\cdot 25^2$}
  \GGGput[    0,  160]( 31.68,530.58){$\omega = 25,~~ h\omega = 2$}
 \end{picture}
\caption{Same experiment as for Figure~\ref{fig:fpu_ho_ttt}, but with the
IMEX method (\ref{trigo-alpha}), $\alpha = 1/4$, and with $\omega = 25$.
The upper picture corresponds to a step size
such that $h\omega =\sqrt{12}$,
i.e., $h\widetilde\omega = 2\pi/3$, the lower picture
such that $h\omega =2$,
i.e., $h\widetilde\omega = \pi/2$.}\label{fig:fpu_ho_imex}
\end{figure}

\bigskip\it\noindent
Experiment 3. Multi-frequency example. \rm
We consider the oscillatory differential equation (\ref{odec}) with $\ell =2$,
$\omega_1 = \omega$, $\omega_2 = \sqrt 2 \,\omega$, and quadratic potential
\[
 U(\bfq ) = 0.01\,q_1\, q_2 .
\]
We apply two trigonometric methods with initial values
$\bfq (0) = (0, 0.3\varepsilon, 0.8\varepsilon )$,
$\dot \bfq (0) = (0, 0.6, 0.7)$ with $\varepsilon = \omega^{-1}$
and step size
$h= 2\pi /(\omega_1 + \omega_2 )$.
The result can be seen in Figure~\ref{fig:desgaut}.
For the symplectic method (B), with filter functions given by
(\ref{deuflhard}), we observe a linear growth
$\bigo (t\varepsilon )$ in the numerical oscillatory
energy, which soon turns into a quadratic growth.
For method (A)  of  \cite[p.\,481]{hairer06gni} (the non-symplectic Gautschi method), with filter functions $\phi (\xi )=1$ and $\psi (\xi ) = \sinc ^2 (\xi /2 )$,
the deviation in the oscillatory
energy is much smaller.

\begin{figure}[t]
 \begin{picture}(0,0)
  \epsfig{file=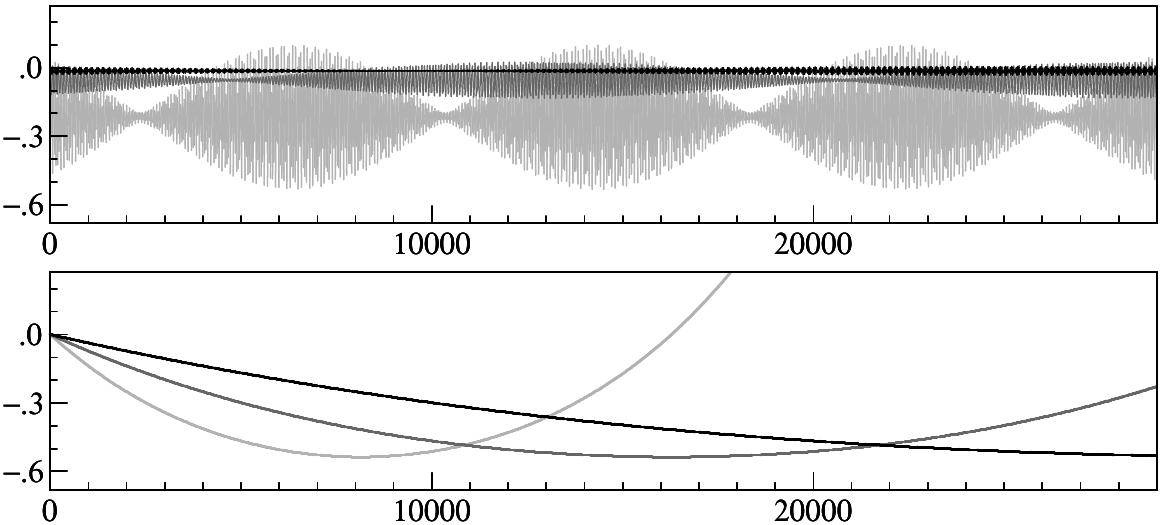}
 \end{picture}%
\begin{picture}(335.6,150.6)(  -7.1, 441.1)
  \GGGput[    0,  100]( 11.68,586.59){$\times \,6\cdot 10^{-5}$}
  \GGGput[    0,  100]( 88.17,586.59){method (A)}
  \GGGput[    0,  100]( 88.17,509.74){method (B)}
 \end{picture}
\caption{Multi-frequency example:
deviation of the oscillatory energy as a function of time $t$ for two trigonometric
methods, $\omega = 200$ (black), $\omega =100$ and $\omega =50$
in gray.}\label{fig:desgaut}
\end{figure}

\section{Numerical gap condition for modified frequencies}
\label{sect:gap}

In order to decide whether a linear combination of frequencies 
is almost-resonant or non-resonant, we will use a gap in the linear 
combinations of frequencies. 
Similarly as in \cite[Sect.~3.1]{gauckler13esi} we construct new
frequencies that satisfy a non-resonance condition outside a resonance module.

For finite $\ell$ we can detect
a gap in the values of $\sin(\frac h2\,  \bfk\cdot\bfomega)$ with small $\|\bfk\|$.

\begin{lemma}\label{lemma-gap}
We fix an arbitrary integer $N\ge 1$ and $0< \delta \le 1/4$.
There exists $0<\mu \le\delta/2$ depending on $N$, $\ell$, and $\delta$,
such that for all $0< h <1$ there
exists $\alpha$ with $\delta/2\le \alpha\le \delta$ depending on the frequencies $\omega_j$
and the step size $h$ such that the set
\[
\Bigl\{\, \Bigl| \sin\Bigl(\sfrac h2\,  \bfk\cdot\bfomega\Bigr) \Bigr| : \|\bfk\|\le N+1 \,\Bigr\}
\]
contains no element in the interval $[h^{1-\alpha+\mu},h^{1-\alpha-\mu}]$.
\end{lemma}

\begin{proof}
This set contains at most $M=\ell^{N+1}$ elements.
Therefore, there exist
$\alpha$ with $\delta/2\le \alpha\le \delta$ depending on the frequencies and the time step size and $0<\mu\le\delta/2 $ depending only on $N$, $\ell$,
and $\delta$, such that this set contains no element in the interval $[h^{1-\alpha+\mu},h^{1-\alpha-\mu}]$.
\qed
\end{proof}

We use this gap to introduce modified frequencies $\varpi_1,\dots,\varpi_\ell$ for
which near-resonant linear combinations, i.e., those taking a value below the gap,
become exactly resonant.
It is convenient to use the notation $\varpi_0 = \omega_0 = 0$.

\begin{lemma}\label{prop-modifiedfrequencies}
Let $N$, $\delta$, and $\mu$ be as in Lemma~\ref{lemma-gap}. Then there exists
$h_0 < 1$ such that for fixed $0<h\le h_0$ satisfying
condition~\eqref{eq-nonres} the following holds:
there
exist modified frequencies $\bfvarpi = (\varpi_1,\dots,\varpi_\ell )$ and a $\Z$-module $\calM\le\Z^d$ such that, with $\alpha$ from Lemma~\ref{lemma-gap},
\begin{align}
\sin\Bigl( \sfrac h2 \,\bfk\cdot\bfvarpi \Bigr) &= 0 
&&\text{for~~$\bfk\in\calM$},\label{eq-lczero}\\
\Big|\sin\Bigl( \sfrac h2\, \bfk\cdot\bfvarpi \Bigr)\Big| &\ge \sfrac{1}{2} h^{1-\alpha-\mu} 
&&\text{for~~$\bfk\notin\calM$~ with ~$\|\bfk\|\le N+1$}.\label{eq-lcbound}
\end{align}
The module $\calM$ does neither contain the unit vectors $\jvec=(0,\dots,1,\dots,0)$ nor their doubles $2\jvec$; if, in addition, the numerical non-resonance condition (\ref{num-nonres}) holds, then the module contains only those $\bfk\in\mathbb{Z}^d$ with $\|\bfk\|\le N+1$ that satisfy $\bfk\cdot\bfvarpi=0$. Moreover, there exists $\gamma>0$ depending only on $N$ and $\ell$ such that
\begin{equation}\label{eq-diff-om-omt}
|\varpi_j-\omega_j| \le \gamma h^{-\alpha+\mu} \qquad\text{for\quad$j=1,\dots,\ell$}.
\end{equation}
\end{lemma}

\begin{proof}
We denote by $\calM$ the $\Z$-module generated by
\begin{equation}\label{setR}
\Bigl\{\, \bfk\in\Z^\ell \, : \, \|\bfk\|\le N+1 \text{ ~and~
$\Bigl|\sin \Bigl(\sfrac h2 \,\bfk\cdot\bfomega \Bigr) \Bigr|
\le h^{1-\alpha+\mu}$} \,\Bigr\}.
\end{equation}
This module is spanned by $d\le \ell$ integer-linearly independent elements $\bfk^1,\ldots,\bfk^d$ (see, for instance, \cite[Chap. III, Theorem~7.1]{lang02alg}). Since the number of modules generated by some subset
$\calR\subseteq\{ \bfk\in \Z^\ell \, : \, \| \bfk \| \le N+1 \}$
depends only on $N$ and~$\ell$, the matrix formed by the
basis vectors as well its pseudo-inverse are bounded by a
constant only depending on $N$ and $\ell$. The same is true
for the coefficients, when the basis vectors are written
as an integer-linear combination of linearly-independent elements of $\calR$.

For the basis vectors
we choose $m_i\in\Z$ such that $|\frac h2 (\bfk^i\cdot\bfomega) - \pi m_i| $
is minimal. Since $\bfk^i $ is a linear combination with integer
coefficients of elements in (\ref{setR}) and since
$|\frac h2 (\bfk^i\cdot\bfomega) - \pi m_i| \le
(\pi /2) \bigl|\sin \bigl( \frac h2 (\bfk^i\cdot\bfomega) \bigr)\bigr|$, it follows from the
addition theorem for sine that
$|\frac h2 (\bfk^i\cdot\bfomega) - \pi m_i| \le \gamma_i h^{1-\alpha+\mu}$ with
constants $\gamma_i$ depending only on $N$ and $\ell$.
We determine $\bfvartheta = (\vartheta_1,\dots,\vartheta_\ell ) \in\real^\ell$ as a solution of minimal norm of 
\begin{equation}\label{eq-linsystem-modfreq}
\bfk^i\cdot\bfomega + \bfk^i\cdot\bfvartheta = \frac{2 \pi m_i}{h}, \qquad i=1,\dots,\ell
\end{equation}
and introduce new frequencies $\bfvarpi = (\varpi_1,\dots,\varpi_\ell )$ as
\[
\varpi_j = \omega_j + \vartheta_j, \qquad j=1,\dots,\ell .
\]
These new frequencies are constructed in such a way that~\eqref{eq-lczero} holds. In addition, the solution $\bfvartheta$ of~\eqref{eq-linsystem-modfreq} is bounded as $\|\bfvartheta\| \le \gamma h^{-\alpha+\mu}$ with $\gamma$ depending only on $N$ and $\ell$,
so that~\eqref{eq-diff-om-omt} holds. Moreover, we have for $\bfk\notin\calM$ with $\|\bfk\|\le N+1$ by the choice of the module $\calM$, Lemma~\ref{lemma-gap} and \eqref{eq-diff-om-omt} 
\[
h^{1-\alpha-\mu} < \Bigl| \sin\Bigl(\sfrac h2 \,\bfk\cdot\bfomega\Bigr) \Bigr| \le 
\Bigl|\sin\Bigl(\sfrac h2 \,\bfk\cdot\bfvarpi\Bigr) \Bigr| + \sfrac\gamma 2 (N+1) h^{1-\alpha+\mu} ,
\]
and hence~\eqref{eq-lcbound} holds if $h$ satisfies
$h^{2\mu} \gamma (N+1) \le 1$.

The condition \eqref{eq-nonres} is used to show that neither the unit vectors $\jvec$ nor their doubles $2\jvec$ are contained in $\calM$. Indeed, if $\jvec\in\calM$ or $2\jvec\in\calM$, then by~\eqref{eq-diff-om-omt} and~\eqref{eq-lczero}
\[
h^{1/2} \le \bigl| \sin (h\omega_j )\bigr| \le
 2 \Bigl| \sin \Bigl( \sfrac h2\, \omega_j \Bigr) \Bigr| \le h |\vartheta_j| \le \gamma h^{1-\alpha+\mu} .
\]
This gives a contradiction if $\gamma h^{1/2 - \alpha + \mu} < 1$.

If, in addition, the numerical non-resonance condition (\ref{num-nonres}) holds, then we have for $\bfk\in\calM$ with $\|\bfk\|\le N+1$ and $\bfk\cdot\bfvarpi\ne 0$ that, by~\eqref{eq-diff-om-omt} and~\eqref{eq-lczero},
\[
h^{1/2} \le \min_{0\ne r\in\mathbb{Z}} \Bigl| \sfrac{h}2 \, \bfk\cdot\bfomega - r\pi \Bigr| \le \Bigl|\sfrac{h}2\, \bfk\cdot\bfvartheta \Bigr| + \min_{0\ne r\in\mathbb{Z}} \Bigl| \sfrac{h}2 \, \bfk\cdot\bfvarpi - r\pi \Bigr| \le \sfrac{\gamma}2 (N+1) \gamma h^{1-\alpha+\mu}.
\]
This gives a contradiction if $(N+1) \gamma h^{1/2 - \alpha + \mu} < 2$.
\qed
\end{proof}

\section{Modulated Fourier expansion}
\label{sect:mfe}

We fix an arbitrary integer $N\ge 1$ and $0<\delta\le 1/4$, and we denote by $\alpha$ and $\mu$ the constants of Lemma \ref{lemma-gap}. For fixed step size $0<h\le h_0$ satisfying \eqref{eq-nonres} and with $h_0$ from Lemma \ref{prop-modifiedfrequencies}, we let $\bfvarpi=(\varpi_1,\dots,\varpi_\ell)$ be the modified frequencies and $\calM$ the $\Z$-module of Lemma \ref{prop-modifiedfrequencies}. 
We denote by $\calN$ a set of representatives of the equivalence classes in
$\Z^\ell/\calM$, which are chosen such that for each $\bfk\in\calN$
the norm of $\bfk$ is minimal in the equivalence class $[\bfk]=\bfk+\calM$,
and with $\bfk\in\calN$, also $-\bfk\in\calN$. We further denote 
$\calK = {\{ \bfk \in \calN \,:\, \| \bfk\| \le N+1 \}}$.

We make the approximation ansatz (modulated Fourier expansion)
\begin{equation}\label{mfe}
\bfq_n \approx \sum_{\bfk\in\calK}
\bfz^\bfk (h^{-\alpha}t)\, \e^{\iu (\bfk\cdot\bfvarpi )t} \qquad \hbox{for}\qquad t=nh ,
\end{equation}
and denote the components of $\bfq_n$ by $q_{n,j}$ and those
of $\bfz^\bfk$ by $z_j^\bfk$ for $j=0,\ldots ,\ell$.
We insert the ansatz (\ref{mfe}) into (\ref{trigo}), expand the
right-hand side into a Taylor series around the smooth
function $\bfPhi\bfz^\bfzero (h^{-\alpha}t)$ and compare the coefficients 
of $\e^{\iu(\bfk\cdot\bfvarpi) t}$. This yields
\begin{equation}\label{mfe-equation}
L_j^\bfk z_j^{\bfk} =  2 \bigl( \cos (h\omega_j ) - \cos (h\varpi_j )\bigr) z_j^\bfk
- h^2 \psi (h\omega_j ) \nabla_j^{-\bfk} \calU (\bfPhi\bfz ) + h^2\delta_j^\bfk,
\end{equation}
where we allow for a small defect $\delta_j^\bfk$. Here,
$\nabla_j^{-\bfk} \calU (\bfx )$ denotes the derivative of
\begin{equation}\label{calU}
\calU (\bfx ) = 
U(\bfx^\bfzero ) 
+ \sum_{m=2}^N
\sum_{j_{1},\ldots ,j_{m}=0}^\ell
\sum_{\bfk^1+\ldots +\bfk^m \in\calM}\frac 1 {m!}\, \partial_{j_1}\ldots\partial_{j_m}
U(\bfx^\bfzero )
\bigl( x_{j_{1}}^{\bfk^1},\ldots ,x_{j_{m}}^{\bfk^m} \bigr)
\end{equation}
with respect to $x_j^{-\bfk}$, where 
the last sum is over multi-indices $\bfk^l \in \calK$ with $\bfk^l \ne \bfzero$ and $\|\bfk^1\| + \ldots + \|\bfk^m \| \le N+1$.
The operator $L_j^\bfk $ in (\ref{mfe-equation}) is given by
\begin{equation}\label{LhD}
\hspace{-1mm}
\begin{array}{rcl}
\bigl( L_j^\bfk z_j^{\bfk} \bigr) (\tau ) &=& \displaystyle
\e^{\iu  (\bfk\cdot\bfvarpi )h} z_j^\bfk (\tau + h^{1-\alpha} ) -
2 \cos (h\varpi_j) z_j^\bfk (\tau ) +
\e^{-\iu  (\bfk\cdot\bfvarpi )h} z_j^\bfk (\tau - h^{1-\alpha} ) \\[2mm]
&=& \displaystyle
4s_{\jvec+\bfk}s_{\jvec-\bfk}z_j^{\bfk}  (\tau ) 
+ 2 \iu s_{2\bfk}h^{1-\alpha}
\dot z_j^{\bfk} (\tau )  + c_{2\bfk}h^{2(1-\alpha)}
\ddot z_j^{\bfk}  (\tau ) + \ldots ~.
\end{array}\hspace{-4mm}
\end{equation}
Here, $s_{\bfk}=\sin (\frac{h}{2}\,{\bfk\cdot\bfvarpi})$ and
$c_{\bfk}=\cos (\frac{h}{2}\,{\bfk\cdot\bfvarpi})$,
and the dots on $z_j^{\bfk}$ represent derivatives with respect to the
scaled time $\tau = h^{-\alpha} t$. The higher order terms are
linear combinations of the $r$th derivative of $z_j^{\bfk}$
(for $r\ge 3$) multiplied by
$h^{r(1-\alpha)}$ and containing one of the factors
$s_{2\bfk}$ or $c_{2\bfk}$.

To get initial values we insert the ansatz (\ref{mfe}) into the
relation $\bfq_0 = \bfq (0)$ and into the second formula of (\ref{trigo}).
This yields the equations
(for $j=0,\dots,\ell$) 
\begin{eqnarray}
q_{0,j} &=& \sum_{\bfk\in\calK}  z_{j}^\bfk (0), \nonumber \\
2h \sinc (h\omega_j )p_{0,j} &=&\sum_{\bfk\in\calK}
\Bigl( z_j^\bfk (h^{1-\alpha }) \e^{\iu (\bfk\cdot\bfvarpi )h} -
 z_j^\bfk  (-h^{1-\alpha }) \e^{-\iu (\bfk\cdot\bfvarpi )h} \Bigr) \label{mod-init}\\
 &=&\sum_{\bfk\in\calK} \Bigl( 2\iu s_{2\bfk} z_j^\bfk (0) +
 2 c_{2\bfk} h^{1-\alpha}\dot z_j^\bfk (0) + 
 \iu s_{2\bfk} h^{2(1-\alpha)}\ddot z_j^\bfk (0)+\ldots \Bigr) .\nonumber
\end{eqnarray}

The Taylor series expansions in the relations (\ref{LhD}) and (\ref{mod-init}) are
truncated after $L\ge (N+3)/(1-\alpha )$ terms, such that the remainder is of size
$\bigo (h^{N+1})$.
Since the following analysis requires estimates for the
modulation functions $z_j^\bfk$ and their derivatives on the interval
$\tau \in [0,1]$ (which corresponds to $t\in [0,h^\alpha ]$), we consider
for functions $\bfz = (z_j^\bfk )$ the norms
\begin{equation}\label{norm}
| z_j^\bfk |_{C^r} =  \max_{0\le \tau \le 1} \max_{0\le l \le r}  \,
\Bigl| \frac{\d^l}{\d\tau^l} z_j^\bfk(\tau) \Bigr| , \qquad
\| \bfz \|_{C^r} = 
\sum_{j=0}^\ell
\sum_{\bfk\in \calK} | z_j^\bfk |_{C^r} .
\end{equation}

\begin{theorem}\label{thm:mfe}
For an arbitrarily fixed $N$ and under Assumptions A and B the numerical
solution $\bfq_n$ of (\ref{trigo})
admits an expansion
\begin{equation}\label{mfe-intheorem}
\bfq_n = \sum_{\bfk\in\calK}
\bfz^\bfk (h^{-\alpha}t)\, \e^{\iu (\bfk\cdot\bfvarpi )t} + \bfr_n\qquad \hbox{for}\qquad t=nh
\le h^\alpha ,
\end{equation}
where the coefficient functions $z_j^\bfk (\tau )$ satisfy
$z_j^{-\bfk} = \overline{z_j^\bfk}$ and are bounded by
\begin{eqnarray}
| z_0^\bfzero |_{C^L} &\le & C\label{estimthm-00} \\
| z_j^{\pm\jvec} |_{C^L} &\le & C \omega_j^{-1} , \qquad j=1,\ldots ,\ell\label{estimthm-jj}\\
| z_j^{\bfzero} |_{C^L} &\le & C h^2 |s_{\jvec}|^{-2}|\psi(h\omega_j)|\varepsilon^{2}  , 
\qquad j=1,\ldots ,\ell\label{estimthm-j0}
\end{eqnarray}
and for all other $(j,\bfk )$ by
\begin{eqnarray}
| z_j^\bfk  |_{C^L} &\le & 
C h^2 |s_{\jvec +\bfk} s_{\jvec -\bfk}|^{-1}|\psi(h\omega_j)|\varepsilon^{\| \bfk \|}
\label{estimthm-jk1} \\
| z_j^\bfk  |_{C^L} &\le & 
C h^2 |s_{\jvec +\bfk} s_{\jvec -\bfk}|^{-1}|\psi(h\omega_j)| (1+ |\bfk\cdot\bfvarpi |)^{-1}
\label{estimthm-jk2}
\end{eqnarray}
with some $L\ge (N+3)/(1-\alpha )$.
The functions $z_j^\bfk$ satisfy the equations (\ref{mfe-equation}) with a
defect bounded for $0\le \tau \le 1$ by
\begin{equation}\label{defect-thm}
|\delta_0^\bfk (\tau )| \le C h^{N+1}, \qquad
|\delta_j^\bfk (\tau )| \le C \omega_j^{-1} h^{N} \qquad\hbox{for}~~j=1,\ldots ,\ell .
\end{equation}
For $\,0\le nh \le h^\alpha$ the remainder term
$\bfr_n = (r_{n,j})$ in (\ref{mfe-intheorem}) is bounded by
\begin{equation}\label{remainder-thm}
|r_{n,0}| \le C h^{N+1},\qquad |r_{n,j}| \le C \omega_j^{-1}h^{N} \qquad\hbox{for}\quad 
j =1,\ldots ,\ell .
\end{equation}
The generic constant $C$ is independent of $\varepsilon$ and the frequencies
$\omega_j \ge \varepsilon^{-1}$, but depends on $\ell$, $N$, and the constants in 
Assumptions A and~B.
\end{theorem}

Details of the proof of Theorem~\ref{thm:mfe}
will be given in Section~\ref{sect:proof-mfe} below.
It combines the techniques of \cite{gauckler13esi} for the analytic solution
of (\ref{ode}) and those of \cite{cohen05nec} for the numerical solution.

The second formula of (\ref{trigo}) requires a modulated Fourier expansion for
the derivative approximation.

\begin{corollary}\label{cor:mfe-derivative}
Under the assumptions of Theorem~\ref{thm:mfe} the derivative
approximation of (\ref{trigo}) satisfies,
for $\tau = h^{-\alpha}t $ and $t=nh\le h^\alpha$,
\[
\bfp_n =  \sum_{\bfk\in\calK} \bfw^\bfk (\tau )\,
\e^{\iu (\bfk\cdot\bfvarpi )t}
+ \dot\bfr_{n} ,
\]
where
\[
 \bfw^\bfk (\tau ) = \bigl( 2h\sinc (h\bfOmega ) \bigr)^{-1}\Bigl(
\bfz^\bfk (\tau +h^{1-\alpha})\, \e^{\iu (\bfk\cdot\bfvarpi )h} -
\bfz^\bfk (\tau -h^{1-\alpha})\, \e^{-\iu (\bfk\cdot\bfvarpi )h} \Bigr) 
\]
and
\begin{equation}\label{remainderdot}
|\dot r_{n,0}| \le C h^{N}\quad\hbox{and}\quad
|\dot r_{n,j}| \le C h^{N}/\kappa \quad\hbox{for}\quad j=1,\ldots ,\ell .
\end{equation}
\end{corollary}

\begin{proof}
The estimate for the remainder
\[
\dot\bfr_{n} = \bigl( 2h \sinc (h\bfOmega )\bigr)^{-1} \bigl( \bfr_{n+1} - \bfr_{n-1} \bigr) .
\]
follows from (\ref{remainder-thm}) and Assumption A.
\qed
\end{proof}

\section{Proof of Theorem~\ref{thm:mfe}}
\label{sect:proof-mfe}

Our aim is to construct functions $z^\bfk_j$ ($j=0,\dots,\ell$ and $\bfk\in\calK$)
such that the defect in equations
(\ref{mfe-equation}) is of size $\bigo(h^{N+1})$. For this we truncate the
Taylor series expansions in (\ref{LhD}) and (\ref{mod-init}) after $L\ge (N+3)/(1-\alpha )$ terms,
and we consider a Picard iteration improving the
approximation by a factor $h^\mu$ in every iteration. This requires $M\ge (N+1+\alpha)/\mu$
iterations.

\subsection{Construction of the modulation functions}
\label{sect:construction}

We denote by $\bfz^m=\bigl([z_j^\bfk]^m\bigr)$ the $m$th
iterate and 
distinguish between the following cases:
\begin{enumerate}
\item
For $j=0$ and $\bfk = \bfzero$ the first two terms in the expansion
(\ref{LhD}) disappear and after division by $h^{2(1-\alpha )}$
we iterate with a second order differential equation
(\ref{mfe-equation})
for $[z_{0}^\bfzero]^{m+1}$:
\begin{equation}\label{mod-it-0}
\hspace{-4mm} 
\frac{\d^2[z_{0}^\bfzero]^{m+1}}{\d\tau^2}  +
\Big[ \frac{2h^{2(1-\alpha)}}{4!} \, \frac{\d^4z_{0}^\bfzero}{\d\tau^4} +
\frac{2h^{4(1-\alpha)}}{6!} \, \frac{\d^6z_{0}^\bfzero}{\d\tau^6} + \ldots \Big]^m
 =  -  h^{2\alpha} \, \nabla_0^{-\bfzero} \calU (\bfPhi\bfz^m ) .
\end{equation}
Here and in the following equations the three dots indicate a truncation
of the series after the term corresponding to the $L$th derivative.
The notation $ \nabla_j^{-\bfk} \calU (\bfPhi\bfz^m ) $ should be interpreted so
that all appearing $z_j^\bfk$ (including $\bfk =\bfzero$) are replaced by
their $m$th iterate.
\item
For $j\ne 0$ and $\bfk = \pm\jvec$ the first term in (\ref{LhD}) disappears and
we iterate using a first order differential equation
for~$[z_{j}^{\pm\jvec}]^{m+1}$:
\begin{eqnarray}
&&\hspace{-3.5mm}
  \pm 2\iu\, s_{2\jvec}\,\frac{\d[z_{j}^{\pm\jvec}]^{m+1}}{\d\tau} 
+ \Big[ c_{2\jvec}h^{1-\alpha} \frac{\d^2z_{j}^{\pm\jvec}}{\d\tau^2} \pm
 \iu \,\frac{s_{2\jvec}}3 h^{2(1-\alpha)} \frac{\d^3z_{j}^{\pm\jvec}}{\d\tau^3} + \ldots \Big]^m
\nonumber
\\[1mm]
&& \hspace{0.1cm}
 =  2 h^{\alpha -1} \bigl( \cos (h\omega_j ) - \cos (h\varpi_j )\bigr) [z_j^{\pm\jvec} ]^m
  - h^{1+\alpha } \psi (h\omega_j) \nabla_{j}^{\mp\jvec}\, \calU (\bfPhi\bfz^m ).\qquad
\label{mod-it-jj} 
\end{eqnarray}
\item
In all other cases we iterate with an explicit equation for~$[z_j^\bfk]^{m+1}$:
\begin{eqnarray}
&&\hspace{-2mm}
4\,s_{\jvec+\bfk}s_{\jvec-\bfk}\, [z_{j}^\bfk]^{m+1}
+ \Big[ 2 \iu \,s_{2\bfk} h^{1-\alpha}
\dot z_j^{\bfk}  + c_{2\bfk}h^{2(1-\alpha)}
\ddot z_j^{\bfk}   + \ldots \Bigr]^m
\label{mod-it-jk} 
\\[1mm]
&& \hspace{1.5cm}
 = 
 2\bigl( \cos (h\omega_j ) - \cos (h\varpi_j )\bigr) [z_j^\bfk ]^m
-  h^2 \psi (h\omega_j ) \nabla_j^{-\bfk} \calU (\bfPhi\bfz^m ). 
\nonumber
\end{eqnarray}
\end{enumerate}

We need initial values
$[z_0^\bfzero]^{m+1} (0), \frac{\d}{\d\tau} [z_0^\bfzero]^{m+1} (0)$,
and $[z_j^{\pm \jvec}]^{m+1}(0) $ for $j\ne 0$. Note that at the iteration $m$
the values $\bfz^m (0) = ([z_j^\bfk ]^m) (0)$ together with the derivatives
of $[z_j^\bfk ]^m$ at $\tau =0$ are known.
Extracting the
dominant terms in (\ref{mod-init}), the required initial values are
determined by the equations
\begin{eqnarray}
[z_0^\bfzero ]^{m+1} (0)&=& q_{0,0} -  \sum_{\bfk\ne\bfzero}  [z_{0}^\bfk]^{m} (0),
\nonumber \\[0mm]
 2 h^{1-\alpha}[\dot z_0^\bfzero ]^{m+1} (0) &=& 2h p_{0,0} -
 \sum_{\bfk\ne\bfzero} \Bigl[ 2\iu \,s_{2\bfk} z_0^\bfk +
 2 \,c_{2\bfk} h^{1-\alpha}\dot z_0^\bfk  + 
 \ldots \Bigr]^{m}(0) \quad  \label{mod-init2}\\
 &-& \Bigl[ \frac 2{3!}\, h^{3(1-\alpha )}\frac{\d^3z_{0}^{\bfzero}}{\d\tau^3} +
  \frac 2{5!}\, h^{5(1-\alpha )}\frac{\d^5z_{0}^{\bfzero}}{\d\tau^5} +\ldots \Bigr]^{m} (0) 
\nonumber
\end{eqnarray}
 and for $j=1,\dots,\ell$ by
\begin{eqnarray}
[z_j^\jvec + z_j^{-\jvec} ]^{m+1} (0)&=& q_{0,j} -  
\sum_{\bfk\ne\pm\jvec}  [z_{j}^\bfk]^{m} (0),\nonumber \\[0mm]
 2 \iu \,s_{2\jvec}[z_j^\jvec - z_j^{-\jvec} ]^{m+1} (0) &=& 2h \sinc (h\omega_j )p_{0,j} -
 \sum_{\bfk\ne\pm\jvec} 2\iu \,s_{2\bfk} [z_j^\bfk ]^{m}(0)\label{mod-init3} \\
&-&
 \sum_{\bfk\in \calK}
 \Bigl[ 2 \,c_{2\bfk} h^{1-\alpha}\dot z_j^\bfk + 
\iu \,s_{2\bfk} h^{2(1-\alpha)}\ddot z_j^\bfk 
  +\ldots \Bigr]^{m} (0) . \nonumber
\end{eqnarray}
The starting iterates are chosen for $(j,\bfk )=(0,\bfzero )$ as 
$[z_0^\bfzero]^0(\tau) = q_{0,0}  $, and
$[z_j^\bfk]^0(\tau)=0$ for all other indices $(j,\bfk)$.

\subsection{Bounds for the modulation functions}
\label{sect:bounds-mod}

To get the desired bounds (\ref{estimthm-00})--(\ref{estimthm-jk2}) of Theorem~\ref{thm:mfe} for the coefficient
functions of the modulated Fourier expansion, we first study the
individual terms appearing in equations (\ref{mod-it-0})-(\ref{mod-it-jk}).

The coefficients of the terms in (\ref{mod-it-0}) corresponding
to arguments of the $m$th iterate
are all small because $0<\alpha < 1$.

Next, we consider the coefficients of the terms in (\ref{mod-it-jj}) corresponding
to arguments of the $m$th iterate.
Since $2\jvec \not\in \calM$ by Lemma~\ref{prop-modifiedfrequencies}, it follows from (\ref{eq-lcbound}) that
$|s_{2\jvec}|\ge \frac 12 h^{1-\alpha -\mu }$. This implies that
$| c_{2\jvec} h^{1-\alpha }/s_{2\jvec}|\le 2 h^\mu$ and
$|h^{1+\alpha}\psi (h\omega_j) /s_{2\jvec}|\le 2 C h^{2\alpha +\mu}$, which are small
for positive $\alpha$ and $\mu$. We have
\[
\frac{h^{\alpha -1} \bigl( \cos (h\omega_j ) - \cos (h\varpi_j )\bigr)}{s_{2\jvec}} =
\frac{2 h^{\alpha -1}\sin (h\frac{\omega_j +\varpi_j}2 ) \sin (h \frac{\varpi_j- \omega_j}2)}
{\sin (h\varpi_j )}  ,
\]
which, by $|\sin (h\frac{\omega_j +\varpi_j}2 )|\le |\sin (h\varpi_j)| + |\sin(h\frac{\omega_j-\varpi_j}2)|$ and by (\ref{eq-diff-om-omt}) and (\ref{eq-lcbound}), is seen
to be bounded by $\bigo (h^\mu )$. This implies that, after division by
$2\iu \, s_{2\jvec}$, the coefficients of the terms in (\ref{mod-it-jj}) are small.

Finally, we consider the coefficients of the terms in (\ref{mod-it-jk}).
From the addition formula for sine we have
\begin{equation}\label{usefulest}
\bigg|\frac{s_{2\bfk}}{s_{\jvec +\bfk} s_{\jvec -\bfk}} \bigg| \le
\frac{|s_{\jvec +\bfk}|+| s_{\jvec -\bfk}|}{|s_{\jvec +\bfk} s_{\jvec -\bfk}|} \le
\frac 1{| s_{\jvec -\bfk}|} + \frac 1{| s_{\jvec +\bfk}|} .
\end{equation}
For $\bfk \ne \pm\jvec$, we have $\jvec \pm \bfk \not\in \calM$ so that
the estimate (\ref{eq-lcbound}) can be applied. This implies that, after division
by $s_{\jvec +\bfk} s_{\jvec -\bfk}$, the coefficient
of $[\dot z_j^\bfk]^m$ in (\ref{mod-it-jk}) is of size $\bigo (h^\mu )$, and that
of the second derivative is $\bigo (h^{2\mu})$. Similar computations
show that the last two terms have coefficients of size $\bigo (h^{2\mu})$
and $\bigo (h^{2\alpha + 2 \mu })$, respectively.

Another useful estimate for the following analysis is
\begin{equation}\label{eq-psiest}
|\psi (h\omega_j )| = \Bigl| \frac{\sincws (h\omega_j ) \phi (h\omega_j )}{\sigma (h\omega_j )} \Bigr| \le \frac{C}{h\omega_j },
\end{equation}
which follows from Assumption~B.

We now prove by induction on $m$, for $m=0,1,\ldots , M$,
\begin{eqnarray}
| [z_0^\bfzero ]^m |_{C^r} &\le & C\label{estim-00} \\
| [z_j^{\pm\jvec}]^m |_{C^r} &\le & C \omega_j^{-1} , \qquad j=1,\ldots ,\ell\label{estim-jj}\\
| [z_j^{\bfzero}]^m |_{C^r} &\le & C h^2 |s_{\jvec}|^{-2}|\psi(h\omega_j)|\varepsilon^{2}  , 
\qquad j=1,\ldots ,\ell\label{estim-j0}
\end{eqnarray}
and for all other $(j,\bfk )$
\begin{eqnarray}
| [z_j^\bfk ]^m |_{C^r} &\le & 
C h^2 |s_{\jvec +\bfk} s_{\jvec -\bfk}|^{-1}|\psi(h\omega_j)|\varepsilon^{\| \bfk \|}\label{estim-jk1} \\
| [z_j^\bfk ]^m |_{C^r} &\le & 
C h^2 |s_{\jvec +\bfk} s_{\jvec -\bfk}|^{-1}|\psi(h\omega_j)| (1+ |\bfk\cdot\bfvarpi |)^{-1}
\label{estim-jk2}
\end{eqnarray}
where $r=L(M-m+1)$. By definition of the starting iterates, $[z_0^\bfzero]^0 (\tau)
= q_{0,0}$ is constant
and all other functions vanish, so that the statements hold for $m=0$.

Assuming the bounds to be true
at level $m$, (\ref{estim-j0}) and (\ref{estim-jk1}) follow for the
$(m+1)$th iterate from
the previous bounds on the coefficients of the equation (\ref{mod-it-jk}) and
the fact that every summand in $\nabla_j^{-\bfk} \cal U (\bfPhi \bfz )$
contains factors $z_{j_1}^{\bfk^1},\ldots ,z_{j_m}^{\bfk^m}$ with
$\bfk^1 + \ldots +\bfk^m =\bfk$ modulo $\calM$. The estimate (\ref{estim-jk2})
follows by applying the triangular inequality to
$\bfk\cdot\bfvarpi = \bfk^1\cdot\bfvarpi +\ldots + \bfk^m\cdot\bfvarpi $ and using
\[
1+|\bfk\cdot\bfvarpi | \le (1+ |\bfk^1\cdot\bfvarpi |)\cdot \ldots\cdot (1 + | \bfk^m\cdot\bfvarpi |) .
\]
The same argument applied
to $\omega_j = \bfk^1\cdot\bfvarpi +\ldots + \bfk^m\cdot\bfvarpi $
yields the estimate $| [\dot z_j^{\pm\jvec}]^{m+1} |_{C^r} \le C \omega_j^{-1}$
from equation (\ref{mod-it-jj}), and we get 
$| [\ddot z_0^\bfzero ]^{m+1} |_{C^r} \le Ch^{2\alpha}$ from (\ref{mod-it-0})
using $\alpha \le 1/2$.

In order to derive estimates of the initial values $[z_0^\bfzero ]^{m+1}(0)$, $[\dot z_0^\bfzero  ]^{m+1}(0)$ and $[z_j^{\pm\jvec} ]^{m+1} (0)$, we use that 
\begin{equation}\label{bounds-initial}
q_{0,0} \le C, \quad p_{0,0} \le C, \quad \omega_j q_{0,j} \le C, \quad p_{0,j} \le C \quad \text{for} \quad j=1,\dots,\ell
\end{equation}
by the bounded energy condition (\ref{bound-energy}) and the bounds on the potential and the numerical solution of Assumption A.
From the relations
(\ref{mod-init2}), still using the estimates at level $m$ and in particular $|[z_0^\bfk ]^m |_{C^r}\le C\eps h^{2\alpha}$ for $\bfk\ne\bfzero$, we obtain the bounds
$|[z_0^\bfzero ]^{m+1}(0)|\le C$ and $|[\dot z_0^\bfzero  ]^{m+1}(0)|\le Ch^\alpha$ 
from the condition $\eps\le h/c_0$.
Using (\ref{eq-psiest}) and the condition $\eps\le h/c_0$ of Assumption~A we get
$| [z_j^\bfk]^m |_{C^r}  
\le C \omega_j^{-1}$ and $h^{1-\alpha}| [z_j^\bfk]^m |_{C^r} / |s_{2\jvec} | 
\le C \omega_j^{-1}$ for $j\ne 0$. Using in addition (\ref{usefulest}) we get
$|s_{2\bfk} [z_j^\bfk]^m |_{C^r} / |s_{2\jvec} | 
\le C \omega_j^{-1}$ 
for $j\ne 0$. These estimates yield
$|[z_j^{\pm\jvec} ]^{m+1} (0)|\le C\omega_j^{-1}$ from (\ref{mod-init3}).
The bounds (\ref{estim-00}) and
(\ref{estim-jj}) are finally obtained by integration of the relations
(\ref{mod-it-0}) and (\ref{mod-it-jj}), respectively.

\subsection{Bounds for the defect}
\label{sect:defect}

Here, we prove the estimate (\ref{defect-thm}) for the defect $\delta_j^\bfk$ in Theorem~\ref{thm:mfe}. 
The defect $\delta_j^\bfk (\tau )$ of (\ref{mfe-equation}) for the $M$th iterate
of Section~\ref{sect:construction} is equal to $[\delta_j^\bfk]^M (\tau )$, where
$[\delta_j^\bfk]^m (\tau )$ is the defect
in inserting the $m$th iterate $\bfz^m =([z_j^\bfk]^m)$ into the modulation equations (\ref{mfe-equation}).
It satisfies for $(j,\bfk )=(0,\bfzero )$
$$
[\delta_0^\bfzero]^m= -h^{-2\alpha}\biggl(  \frac{\d^2[z_{0}^\bfzero]^{m+1}}{\d\tau^2} - \frac{\d^2[z_{0}^\bfzero]^{m}}{\d\tau^2} \biggr) ,
$$
for $j=1,\dots,\ell$ and $\bfk = \pm \jvec$
\[
[\delta_j^{\pm\jvec}]^m = \mp2\iu \, s_{2\jvec}\, h^{-1-\alpha}\biggl(
\frac{\d[z_{j}^{\pm\jvec}]^{m+1}}{\d\tau} - \frac{\d[z_{j}^{\pm\jvec}]^{m}}{\d\tau}\biggr) ,
\]
and for all other $(j,\bfk)$
$$
[\delta_j^\bfk]^m =  -\frac {4}{h^2}\, s_{\jvec +\bfk} s_{\jvec -\bfk}\, 
\Bigl( [z_{j}^\bfk]^{m+1} - [z_{j}^\bfk]^{m} \Bigr).
$$
With the notation
$$
\bfLambda\bfz = \bfv = (v_j^\bfk) \quad\hbox{with} \quad \left\{
\begin{array}{l}
v_0^\bfzero = h^{-2\alpha}z_0^\bfzero
\\[1mm]
v_0^\bfk = 4 \, s_{\bfk}^2 h^{-2} z_0^\bfk \quad\ \bfk \ne\bfzero
\\[1mm]
v_j^{\pm\jvec} = 2\, h\omega_j\,
s_{2\jvec}\, h^{-1-\alpha} z_j^{\pm\jvec}, \quad\  j=1,\dots,\ell
\\[1mm]
v_j^\bfk = 4 \,h\omega_j \, s_{\jvec +\bfk} s_{\jvec -\bfk} h^{-2} z_j^\bfk \quad\ \hbox{else,}
\end{array}
\right.
$$
and the above formulas for the defect we have in
the norm (\ref{norm}) that
\[
 \bigl|[\delta_0^\bfk ]^m\bigr| \le  \| \bfv^{m+1} - \bfv^m \|_{C^2}, \qquad
 \bigl|[\delta_j^\bfk]^m\bigr| \le  (h\omega_j)^{-1} \| \bfv^{m+1} - \bfv^m \|_{C^1}
\]
for $j=1,\ldots ,\ell$ and all $\bfk\in\calK$.
We therefore study $\bfv^{m+1} - \bfv^m$ and show by induction on $m$ that 
for $m\le M$ and $r=L(M-m+1)$,
\begin{equation}\label{eq:induction}
\| \bfv^{m+1} - \bfv^m \|_{C^r} = \bigo(h^{m\mu-\alpha}),
\end{equation}
which implies the estimate \eqref{defect-thm} of the defect in Theorem~\ref{thm:mfe} if we use $M\ge (N+1+\alpha)/\mu$ iterations for the construction of the modulation functions.
In the following estimates we repeatedly use (\ref{eq-psiest}) to obtain the factor $\omega_j^{-1}$
where needed.
For $m=0$, the definition of $\bfv^0$ and the bounds
of Section~\ref{sect:bounds-mod} yield
$| [v_j^\bfk ]^1 - [v_j^\bfk ]^0 |_{C^{LM+L}} = \bigo (h^{-\alpha})$. 

For the induction proof
we first consider the functions $[v_j^\bfk]^m$ defined in (\ref{mod-it-jk}). The bounds
of Section~\ref{sect:bounds-mod} yield
\[
\bigl| [v_j^\bfk ]^{m+1} - [v_j^\bfk]^{m} \bigr|_{C^r} \le C h^\mu \| \bfv^m - \bfv^{m-1} \|_{C^{r+L}}.
\]
We next consider the diagonal elements $[v_j^{\pm\jvec}]^m$. From (\ref{mod-init3}) we obtain that
$$
\bigl| [v_j^{\pm\jvec}]^{m+1}(0) - [v_j^{\pm\jvec}]^{m}(0) \bigr| \le Ch^\mu \| \bfv^m - \bfv^{m-1} \|_{C^L}.
$$
Integration of equation (\ref{mod-it-jj})
then yields
$$
\bigl| [v_j^{\pm\jvec}]^{m+1}(\tau) - [v_j^{\pm\jvec}]^{m}(\tau) \bigr| \le C h^\mu \| \bfv^m - \bfv^{m-1} \|_{C^L}, \quad 0\le\tau\le1 .
$$
Repeated differentiation in (\ref{mod-it-jj}) further shows that 
$$
\bigl| [v_j^{\pm\jvec}]^{m+1} - [ v_j^{\pm\jvec}]^{m} \bigr|_{C^r} \le 
C h^\mu \| \bfv^m - \bfv^{m-1} \|_{C^{r+L}}.
$$
Using $2(1-\alpha ) \ge \mu$ and $2\alpha\ge \mu$ we also bound
$$
\bigl| [v_0^\bfzero]^{m+1} - [v_0^\bfzero]^{m} \bigr|_{C^r} \le C h^\mu \| \bfv^m - \bfv^{m-1} \|_{C^{r+L}}.
$$
Summarizing we get
$$
\| \bfv^{m+1} - \bfv^m \|_{C^r} \le C h^\mu \| \bfv^{m} - \bfv^{m-1} \|_{C^{r+L}} .
$$
This proves \eqref{eq:induction} and the estimate of the defect of Theorem~\ref{thm:mfe}.

\subsection{Solution approximation}

In this section we prove the bounds (\ref{remainder-thm}), which then completes the
proof of Theorem~\ref{thm:mfe}.
We consider the $M$th iterates of the modulation functions
(with $M \ge (N+1+\alpha)/\mu$) and omit the superscript $M$ 
on the modulation functions and their defects. 
The truncated modulated Fourier expansion
\begin{equation}\label{qntilde}
\widetilde \bfq_n= \sum_{\bfk\in\calK} \bfz^\bfk (h^{-\alpha}t)\, \e^{\iu (\bfk\cdot\bfvarpi )t},
\qquad t=nh ,
\end{equation}
inserted into the method (\ref{trigo})
$$
\widetilde \bfq_{n+1} - 2\cos (h\bfOmega ) \, \widetilde \bfq_n + \widetilde \bfq_{n-1} =
- h^2 \bfPsi \nabla U (\bfPhi \widetilde \bfq_n ) + h^2 \bfd_n
$$
has a defect $\bfd_n$. The $j$th component of the defect is given by
\begin{equation}\label{defect}\begin{split}
d_{n,j} &= \sum_{\bfk\in \calK} \bigl(\delta_j^\bfk(h^{-\alpha}t) - \widehat{\rho}_j^\bfk(h^{-\alpha}t) \bigr) \e^{\iu(\bfk\cdot\bfvarpi)t}\\
 &\qquad - 
\psi (h\omega_j)\Bigl(\, \sum_{\bfk\in\calN} \nabla_j^{-\bfk}\widehat{\calU}(\bfPhi\bfz(h^{-\alpha}t))
\e^{\iu(\bfk\cdot\bfvarpi)t} + \rho_{n,j}\Bigr),
\end{split}\end{equation}
where  $\delta_j^\bfk$ is defined in (\ref{mfe-equation}), $\widehat{\rho}_j^\bfk$ denotes the remainder term of the truncated Taylor series expansions in (\ref{LhD}), and $(\rho_{n,j})$
denotes  the remainder term in the truncated Taylor series expansion of the
gradient of $U$ around $\bfPhi\bfz^\bfzero(h^{-\alpha}t)$.
The second (finite) sum collects those terms that where neglected in the definition of $\calU$ in (\ref{calU}), i.e., $\widehat{\calU}$ is defined as $\calU$ in (\ref{calU}) but with the last sum over $\bfk^l\in\calN$ with $\bfk^l\ne\bfzero$ and $\|\bfk^1\|+\dots+\|\bfk^m\|>N+1$.
Writing $\widetilde \bfq_n = \bfz^\bfzero (h^{-\alpha}t)
+\widehat \bfq_n$, and
omitting the index $n$ and the argument $\tau=h^{-\alpha}t$ we have
\[
\rho_{n,j}= \nabla_j U(\bfPhi\widetilde\bfq) - \nabla_j U(\bfPhi\bfz^\bfzero ) 
- \sum_{m=1}^N
\sum_{j_{1},\ldots ,j_{m}=0}^n
 \frac 1 {m!}\, \partial_{j_1}\ldots\partial_{j_m}\nabla_j U(\bfPhi\bfz^\bfzero )
\bigl( \widehat q_{j_{1}},\ldots ,\widehat q_{j_{m}} \bigr).
\]
In the following we work with the weighted norm
\[
\| \bfd_n \|_\bfomega = |d_{n,0}| +\sum_{j=1}^\ell h\omega_j | d_{n,j} | .
\]
On $\tau$-intervals of length $1$, which corresponds to $t=nh\le h^\alpha$,
the first term in~(\ref{defect}) is $\bigo (h^{N+1})$ by definition of the
modulation functions (see Section~\ref{sect:defect}), and the $j$th component
of the other two terms contains a factor $\psi (h\omega_j )\varepsilon^{N+1}$
as a consequence of the appearance of sufficiently many factors of
$z_j^\bfk$. By Assumption A ($h/\varepsilon \ge c_0 >0$) and (\ref{eq-psiest}) 
this implies that
\[
\|\bfd_n\|_\bfomega = \bigo (h^{N+1}) \qquad \hbox{for} \qquad nh\le h^\alpha . 
\]

As for the defect in Section~\ref{sect:defect} it follows from (\ref{mod-init2})
and (\ref{mod-init3}) that
\[
\begin{array}{rcl}
\widetilde q_{0,0} - q_{0,0} &=& [z_0^\bfzero ]^{M}(0) -  [z_0^\bfzero ]^{M+1}(0) \\[2mm]
\widetilde q_{0,j} - q_{0,j} &=& [z_j^{\jvec} + z_j^{-\jvec}]^{M}(0) -  
[z_j^{\jvec} + z_j^{-\jvec}]^{M+1}(0) .
\end{array}
\]
The estimate (\ref{eq:induction}) for $m=M$
proves that $\| \widetilde \bfq_0 - \bfq_0 \|_\bfomega = \bigo (h^{N+1})$.
Similarly, for the truncated derivative approximation
\[
\widetilde \bfp_n =\sum_{\bfk\in\calK}
\bfw^\bfk (\tau ) \e^{\iu (\bfk\cdot\bfvarpi )t}, \qquad t=nh
\]
with components $\widetilde p_{n,j}$ (see
Corollary~\ref{cor:mfe-derivative}), we have (up to an
error of size $\bigo (h^{N+3})$)
\[
\begin{array}{rcl}
h ({\widetilde p}_{0,0} - p_{0,0}) &=& h^{1-\alpha} \bigl( [\dot z_0^\bfzero ]^{M}(0) 
-  [\dot z_0^\bfzero ]^{M+1}(0) \bigr) \\[2mm]
h \sinc (h\omega_j ) ({\widetilde p}_{0,j} - p_{0,j}) &=& \iu \,s_{2\jvec}
\bigl( [z_j^{\jvec} - z_j^{-\jvec}]^{M}(0) -  
[z_j^{\jvec} - z_j^{-\jvec}]^{M+1}(0)\bigr) .
\end{array}
\]
The estimate (\ref{eq:induction}) and the relation (\ref{trigo-q1}) for $\bfq_1$
yield the bound $\|\widetilde \bfq_1 - \bfq_1 \|_\bfomega = \bigo (h^{N+1})$.
We have used that the nonlinearity $\bfPsi \nabla U (\bfPhi \bfq )$ is Lipschitz-continuous
in the norm $\|\cdot \|_\bfomega$
with a constant that only depends on bounds of the derivatives of the potential $U$.
Using a discrete Gronwall Lemma, a standard analysis of the propagation of
errors in the method (\ref{trigo}) (see \cite[Section XIII.4.1]{hairer06gni}) then proves
the bound of $\bfr_n =\bfq_n - \widetilde \bfq_n$ for $0\le nh \le h^\alpha$ as stated in
Theorem~\ref{thm:mfe}.

\section{Almost-invariants of the modulation system}
\label{sect:invariants}

In this section we show that the system for the modulation functions
has two almost-invariants -- one is related to the slow energy $H_{\rm slow}(\bfp ,\bfq )$
and the other to the oscillatory energy $H_\bfomega (\bfp ,\bfq )$.

\subsection{Almost-invariant related to the slow energy}
\label{sect:invariant1}

We multiply the equation (\ref{mfe-equation}) by
$\phi (h\omega_j ) (\dot z_j^{-\bfk})^\T$ and sum over all $j\in \{0,\ldots ,\ell\}$
and $\bfk\in \calK$ to obtain
\begin{equation}\label{eq-invariantH}
\begin{array}{rcl}
\displaystyle\frac {h^{-\alpha}}{h^2}\sum_{j=0}^\ell\sum_{\bfk\in\calK}
\frac{\phi (h\omega_j )}{\psi (h\omega_j )} \Bigl( (\dot z_j^{-\bfk})^\T L_j^\bfk z_j^{\bfk} &-&  
2 \bigl( \cos (h\omega_j ) - \cos (h\varpi_j )\bigr) 
\displaystyle (\dot z_j^{-\bfk})^\T z_j^\bfk \Bigr)\\[0mm]
&& \hspace{-25mm}\displaystyle = -~  h^{-\alpha}\frac{\d}{\d \tau} \calU (\bfPhi\bfz ) + 
h^{-\alpha} \sum_{j=0}^\ell\sum_{\bfk\in\calK}
\frac{\phi (h\omega_j )}{\psi (h\omega_j )}  (\dot z_j^{-\bfk})^\T\delta_j^\bfk.
\end{array}
\end{equation}
As in \cite[page 508]{hairer06gni} the left-hand side of this equation is seen
to be a total differential. Therefore, there exists a function
$\calE [\bfz ] (t )$, which depends on the values at $\tau = h^{-\alpha}t$ of
the function $\bfz$ and of its first $L$ derivatives, such that
\begin{equation}\label{dtH}
\frac{\d}{\d t } \calE [\bfz ] (t ) = \bigo (h^{N+2-\alpha}/\kappa ) =
\bigo (h^{N+1}) .
\end{equation}
Here, we have used the bounds of Theorem~\ref{thm:mfe} for the defect
$\delta_j^\bfk$ and for the $z_j^\bfk$, and the estimate
\begin{equation}\label{phidurchpsi}
\Bigl| \frac{\phi (h\omega_j )}{\psi (h\omega_j )}\Bigr| =
\Bigl| \frac{\sigma (h\omega_j )}{\sincws (h\omega_j )}\Bigr| \le \frac {C_1 h\omega_j}{\kappa} ,
\end{equation}
which follows from Assumptions A and B.

\begin{theorem}\label{thm:invariant1}
In the situation of Theorem~\ref{thm:mfe} we have for $0\le t=nh \le h^\alpha$
\begin{eqnarray*}
 \calE [\bfz ] (t ) &=&  \calE [\bfz ] (0 ) +  \bigo ( t  h^{N+1}) \\[1mm]
 \calE [\bfz ] (t ) &=& H_{\rm slow} (\bfp_n, \bfq_n ) +
 \bigo (\varepsilon h^{-\alpha}  )+ \bigo (h^{2(1-\alpha )}).
\end{eqnarray*}
\end{theorem}

\begin{proof}
The first statement follows by integration of (\ref{dtH}).
We next show that
\begin{equation}\label{mod-Hslow}
 \calE [\bfz ] (t ) = \sfrac 12 |h^{-\alpha}\dot z_0^\bfzero (\tau )|^2 + U(\bfPhi \bfz^\bfzero ) +
 \bigo (\varepsilon h^{-\alpha}  )+ \bigo (h^{2(1-\alpha )}).
\end{equation}
By definition (\ref{calU}) of $\calU$ and the estimates (\ref{estimthm-00})--(\ref{estimthm-jk2}) on the modulation functions
we have $\calU (\bfPhi \bfz ) = U(\bfPhi \bfz^\bfzero ) + \bigo (\varepsilon^2 )$.
The term with $j=0$ and $\bfk =\bfzero$ in (\ref{eq-invariantH}) yields
$\frac 12 |h^{-\alpha}\dot z_0^\bfzero |^2 + \bigo (h^{2(1-\alpha )})$. 
The term with $j=0$ and $\bfk\ne\bfzero$ gives $\bigo (\varepsilon^2 )$
because of the estimates (\ref{eq-lcbound}) and (\ref{estimthm-jk1}).
For $j>0$ and $\bfk = \pm\jvec$ the dominant term is
\begin{eqnarray*}
&&\hspace{-8mm}\frac 1{h^2} \frac{\phi (h\omega_j )}{\psi (h\omega_j )}
\Bigl( \cos (h\omega_j ) - \cos (h\varpi_j )\Bigr) | z_j^{\pm \jvec}|^2\\
&& = \frac 1{h^2} \frac{\sigma (h\omega_j ) h \omega_j}{\sin (h\omega_j )}
\, 2 \sin \Bigl(\frac{h(\omega_j+\varpi_j )}2\Bigr) 
\sin \Bigl(\frac{h(\omega_j-\varpi_j )}2\Bigr) | z_j^{\pm \jvec}|^2\\
&& = \bigo (h^{-2} h \omega_j h^{1-\alpha+\mu } \omega_j^{-2})
= \bigo (\omega_j^{-1} h^{-\alpha+\mu }) = \bigo (\varepsilon h^{-\alpha }) ,
\end{eqnarray*}
where we have used the bounds (\ref{eq-diff-om-omt}) and
(\ref{estimthm-jj}).
All further terms are smaller. This proves
(\ref{mod-Hslow}).

To relate the right-hand side of (\ref{mod-Hslow}) to the slow energy
$H_{\rm slow} (\bfp_n, \bfq_n )$
we use the modulated Fourier expansions from Theorem~\ref{thm:mfe}
and Corollary~\ref{cor:mfe-derivative},
and the first estimate of Lemma~\ref{lem:deriv} below.
We note the bound $\bfPhi \bfz^\bfzero - \bfz^\bfzero = \bigo (\varepsilon^2 )$,
which gives us $U (\bfPhi \bfz^\bfzero ) - U(\bfq_n ) = \bigo (\varepsilon )$.
This yields
\[
H_{\rm slow} (\bfp_n, \bfq_n ) =
\sfrac 12 |h^{-\alpha}\dot z_0^\bfzero |^2 + U(\bfPhi \bfz^\bfzero ) +
 \bigo (\varepsilon) ,
\]
which together with (\ref{mod-Hslow}) proves the result.
\qed
\end{proof}

\begin{lemma}\label{lem:deriv}
For $j=0$ we have
\[
p_{n,0} = h^{-\alpha} \dot z_0^\bfzero (\tau ) + \bigo (h^{2-3\alpha}) + 
\bigo (\varepsilon h^{\alpha} ) .
\]
For $j=1,\ldots ,\ell$ we have
\[
p_{n,j} = \iu \, \frac{\sin (h\varpi_j )}{\sin (h \omega_j )}\, \omega_j
\Bigl( z_j^{\jvec} (\tau )  -
z_j^{-\jvec} (\tau )\Bigr)  + \bigo (h^{1-\alpha}/\kappa ) .
\]
\end{lemma}

\begin{proof}
The estimates follow from Corollary~\ref{cor:mfe-derivative} and the
bounds from Theorem~\ref{thm:mfe}.
We use the estimate
(\ref{usefulest}) of Section~\ref{sect:bounds-mod}, and (\ref{eq-lcbound})
to bound the factors $s_{\jvec\pm \bfk}$ from below.
\qed
\end{proof}

\subsection{Almost-invariant related to the oscillatory energy}
\label{sect:invariant2}

Since the sum in the definition of $\calU (\bfz )$ is over multi-indices
$\bfk^1,\ldots ,\bfk^m$ with $\bfk^1+\ldots +\bfk^m \in\calM$ and $\|\bfk^1\| + \ldots + \|\bfk^m \| \le N+1$, we have under condition (\ref{num-nonres}) by Lemma \ref{prop-modifiedfrequencies} that then $(\bfk^1+\ldots +\bfk^m)\cdot\bfvarpi=0$, and therefore
\[
\calU (S (\theta )\bfPhi\bfz ) = \calU (\bfPhi\bfz ), \qquad S (\theta )\bfx =
(\e^{\iu (\bfk\cdot\bfvarpi )\theta} x_j^\bfk ) .
\]
Differentiating this relation with respect to $\theta$ yields
\[
0 = \frac{\d}{\d\theta}\Big|_{\theta =0} \calU (S (\theta )\bfPhi\bfz ) =
\sum_{j=0}^\ell \sum_{\bfk\in\calK} \iu (\bfk\cdot\bfvarpi ) \phi (h\omega_j)
(z_j^\bfk )^\T \nabla_j^\bfk \calU (\bfPhi\bfz )  .
\]
Similar as before we multiply the equation (\ref{mfe-equation}) by
$\phi (h\omega_j )(-\bfk\cdot\bfvarpi )(z_j^{-\bfk})^\T$ and sum over all $j\in \{0,\ldots ,\ell\}$
and $\bfk\in \calK$ to obtain
\begin{eqnarray}
 -\frac{\iu}{h^2}\sum_{j=0}^\ell\sum_{\bfk\in\calK}
\frac{\phi (h\omega_j )}{\psi (h\omega_j )} (\bfk\cdot\bfvarpi )
 \Bigl( (z_j^{-\bfk})^\T L_j^\bfk z_j^{\bfk} &-&  
2 \bigl( \cos (h\omega_j ) - \cos (h\varpi_j )\bigr) 
(z_j^{-\bfk})^\T z_j^\bfk \Bigr)\nonumber \\[-2mm]
&& \hspace{-30mm} = ~ - \iu~
\sum_{j=0}^\ell\sum_{\bfk\in\calK} (\bfk\cdot\bfvarpi )
\frac{\phi (h\omega_j )}{\psi (h\omega_j )}  (z_j^{-\bfk})^\T\delta_j^\bfk .\nonumber
\end{eqnarray}
The coefficients of the terms $(z_j^{-\bfk})^\T z_j^\bfk$  and 
$(z_j^{\bfk})^\T z_j^{-\bfk}$ in this expression have opposite sign and therefore
cancel in the sum. Consequently,
as in Section~\ref{sect:invariant1}, the formulas
of \cite[page 508]{hairer06gni} show that the left-hand
expression is a total differential.
Therefore, there exists a function
$\calI [\bfz ] (t )$, which depends on the values at $\tau = h^{-\alpha }t$ of
the function $\bfz$ and of its first $L$ derivatives, such that
\begin{equation}\label{deriv-invariant1}
\frac{\d}{\d t } \calI [\bfz ] (t ) = \bigo ( h^{N+1}/\kappa )  =  \bigo ( h^{N}) .
\end{equation}
For this estimate we use the bounds (\ref{estimthm-jj}) and (\ref{estimthm-jk2})
for $(\bfk\cdot\bfvarpi ) z_j^\bfk $, the bound
(\ref{defect-thm}) for $\delta_j^\bfk$, and the estimate (\ref{phidurchpsi}).

\begin{theorem}\label{thm:invariant2}
In the situation of Theorem~\ref{thm:mfe} we have under condition (\ref{num-nonres}) for $0\le t=nh \le h^\alpha$
\begin{eqnarray*}
 \calI [\bfz ] (t ) &=&  \calI [\bfz ] (0 ) +  \bigo ( t  h^{N}) \\[1mm]
 \calI [\bfz ] (t ) &=& H_{\bfomega}^* (\bfp_n, \bfq_n ) + \bigo (h^{1-\alpha}/\kappa ) ,
\end{eqnarray*}
where $H_{\bfomega}^* $ is defined in (\ref{mod-energies}).
\end{theorem}

\begin{proof}
The first statement follows by integration of (\ref{deriv-invariant1}).
The dominant term of $\calI [\bfz ] (t )$ is that for $\bfk = \pm\jvec$ with
the lowest derivative. With
$\sigma (\xi )$ from (\ref{sigma}) it
is given by
\[
\sum_{j=1}^\ell \sigma (h\omega_j) 2 \omega_j^2
 \biggl( \frac{\varpi_j}{\omega_j} \frac{\sin (h\varpi_j)}{\sin (h\omega_j)} \biggr)
  \bigl| z_j^{\jvec} (h^{-\alpha }t ) \bigr|^2 .
\]
Using (\ref{eq-diff-om-omt}),
the expression in brackets is seen to be of the form $1+ \bigo (h^{1-\alpha}/\kappa )$. 
All other terms are at most of size $\bigo (h^{1-\alpha+\mu}/\kappa )$.
\qed
\end{proof}

\section{Proof of Theorem~\ref{thm:main}}
\label{sect:proof-thm12}

\subsection{Transition from one interval to the next}

Theorems \ref{thm:invariant1} and \ref{thm:invariant2} are only valid on
a short time interval of length $\nu h \le h^\alpha$. 
Here, we consider the modulated Fourier expansion corresponding to
starting values $(\bfq_\nu , \bfp_\nu )$ and compare the almost-invariants
to those corresponding to $(\bfq_0 ,\bfp_0 )$.

\begin{lemma}\label{lem:trans}
In the situation of Theorem~\ref{thm:mfe}, let $z_j^\bfk (\tau )$ be the
coefficient functions of the modulated Fourier expansion for
initial data $(\bfq_0, \bfp_0)$. We let $\widetilde z_j^\bfk (\tau )$ be
the coefficient functions corresponding to $(\bfq_\nu , \bfp_\nu )$ for
$\nu$ with $\nu h \le h^\alpha$. Then, 
\begin{eqnarray*}
\calE [\bfz ] (\nu h)  &=& \calE [\widetilde \bfz ] (0) + \bigo (h^{N-1}) \\[1mm]
\calI [\bfz ] (\nu h)  &=& \calI [\widetilde \bfz ] (0) + \bigo (h^{N-1}) .
\end{eqnarray*}
\end{lemma}

\begin{proof}
Let $\bfz^M = ([z^\bfk_j]^M)$ be the last iterate in the construction of the
modulation function $\bfz$ for initial data $(\bfq_0, \bfp_0)$, and
let $\widetilde{\bfz}^m = ([\widetilde{z}^\bfk_j]^m)$
be the $m$th iterate of the modulation function~$\widetilde{\bfz}$
corresponding to initial data $(\bfq_\nu , \bfp_\nu )$.
We aim in estimating the difference $\Delta \bfz^m (\tau) =
\bfz^M (\nu h^{1-\alpha}+\tau) - \widetilde{\bfz}^m(\tau)$.
The functions $[\widetilde z_j^\bfk ]^m$ satisfy the relations of Section~\ref{sect:construction}
with $(\bfq_0, \bfp_0)$ replaced by $(\bfq_\nu, \bfp_\nu)$.
The functions $[z^\bfk_j]^M$ satisfy the same relations, where the superscripts
$m$ and $m+1$ are changed to $M$ and the defect $\delta_j^\bfk$ is
added (see the formulas of Section~\ref{sect:defect}).
For the rescaled differences
$\Delta\bfv^m(\tau) = \bfLambda\Delta\bfz^m (\tau )$ the same arguments as in
Section~\ref{sect:defect} yield
\begin{equation}\label{interface}
\| \Delta\bfv^{m+1}\|_{C^r} \le C h^\mu
\| \Delta\bfv^{m} \|_{C^{r+L}} + D h^{N-\alpha}/\kappa
\end{equation}
with $r=L(M-m)$. The bounds (\ref{defect-thm}) for the defect introduce an
inhomogeneity of size $\bigo (h^{N+1})$ in (\ref{interface}), whereas the bounds
(\ref{remainder-thm}) for the difference between $\bfq_n$ and $\widetilde\bfq_n$
of (\ref{qntilde}) and those of (\ref{remainderdot}) for the
derivative approximations introduce an inhomogeneity of size $\bigo (h^{N-\alpha}/\kappa)$.
From (\ref{interface}) it follows by induction on $m$ that
\[
\| \Delta\bfv^{m}\|_{C^{L(M-m+1)}} \le (Ch^\mu )^{m}
\|\Delta\bfv^0\|_{C^{L(M+1)}} + mDh^{N-\alpha}/\kappa .
\]
Since the bounds of Theorem~\ref{thm:mfe}
give $\|\Delta\bfv^0\|_{C^{L(M+1)}} = \bigo (h^{-\alpha})$, we obtain after
$M\ge (N+\alpha )\mu$ iterations that 
$\|\Delta\bfv^M\|_{C^L} = \bigo (h^{N-\alpha}/\kappa)$ and consequently also
\[
\| \Delta \bfz ^M \|_{C^L} = \bigo (h^{N-1}) .
\]
Using Lipschitz estimates for $\calE$ and $\calI$ yields the result.
\qed
\end{proof}

\subsection{From short to long time intervals}

We put the estimates for many short time intervals of length
$\nu h\le h^\alpha$
together to get the long-time result of Theorem~\ref{thm:main}. For $m=0,1,2,\dots$, let $\bfz_m(\tau )$ collect the coefficient
functions of the modulated Fourier expansion starting from
$(\bfq_{m\nu}, \bfp_{m\nu})$. 
Since we aim at proving that the modified oscillatory energy $H_\bfomega^*$
remains nearly constant, we consider,
instead of the bounded energy assumption
(\ref{bound-energy}), the condition
\begin{equation}\label{bound-energy2}
\bigl| H_\bfomega^* \bigl(\bfp (0) ,\bfq (0)\bigr) \bigr| + \bigl| H_{\rm slow} \bigl(\bfp (0) ,\bfq (0)\bigr) \bigr| \le E^*.
\end{equation}
This condition follows with $E^*=(C_1+1)E + (C_1+2)\widehat{K}$ from Assumptions A and B, where $\widehat{K}$ denotes the bound of the potential $U$ on the set $K_\rho$ of Assumption A. 
Since $|\sigma (h\omega_j )|\ge c_1 >0$
and all $\sigma (h\omega_j )$ have the same sign,
we can use (\ref{bound-energy2}) instead of (\ref{bound-energy}) in the proof of
Theorem~\ref{thm:mfe} (estimates (\ref{bounds-initial})).

As long as
(\ref{bound-energy2}) holds with $2E^*$ instead of $E^*$, Theorem~\ref{thm:invariant2}
yields for $0\le nh \le h^\alpha$
$$
\bigl|\calI [ \bfz_m] (nh) - \calI [ \bfz_m] (0) \bigr| \le C nh^{N+1}.
$$
By Lemma~\ref{lem:trans},
$$
\bigl| \calI [\bfz_m] (\nu h)-\calI [\bfz_{m+1}] (0) \bigr|  \le C h^{N-1}.
 $$
 Summing up these estimates over $m$ and applying the triangle inequality yields, for
 $0\le n \le \nu$,
$$
\bigl|\calI [\bfz_{m}] (nh) -  \calI [\bfz_0 ] (0) \bigr| \le (m+(m\nu +n)h^2)C h^{N-1}.
$$
By Theorem~\ref{thm:invariant2}, we have
$$
\bigl|\calI [\bfz_{m}] (nh) - H_{\bfomega}^* (\bfp_{m\nu +n}, \bfq_{m\nu +n})\bigr| \le
C' h^{1-\alpha}/\kappa .
$$
Combining these bounds we obtain for $t=(m\nu+n)h$ (and $\nu h \approx h^\alpha$)
$$
\bigl|H_{\bfomega}^* (\bfp_{m\nu +n}, \bfq_{m\nu +n})
- H_{\bfomega}^* (\bfp_{0}, \bfq_{0})\bigr| \le 2C t h^{N-1-\alpha} + 2C'
h^{1-\alpha}/\kappa ,
$$
which is $\bigo (h^{1-\alpha}/\kappa )$ for $t\le h^{-N+2}$. In the same way, we obtain for $t=(m\nu+n)h$ (and $\nu h \approx h^\alpha$), using the almost-invariant $\calE$ and Theorem \ref{thm:invariant1} instead of $\calI$ and Theorem \ref{thm:invariant2},
\[
\bigl|H_{\rm slow} (\bfp_{m\nu +n}, \bfq_{m\nu +n})
- H_{\rm slow} (\bfp_{0}, \bfq_{0})\bigr| \le 2C t h^{N-1-\alpha} + 2C'
 \bigl(\eps h^{-\alpha} + h^{2(1-\alpha)} \bigr),
\]
which is $\bigo (\eps h^{-\alpha}) + \bigo (h)$ for $t\le \eps h^{-N+1}$ and $\bigo (h^{1-\alpha})$ for $t\le h^{-N+2}$. 
These estimates ensure that, for sufficiently small step size,
(\ref{bound-energy2}) holds with $2E^*$ instead of $E^*$ on such time intervals.
Replacing the
arbitrary integer $N$ by $N+2$ yields the statements of Theorem~\ref{thm:main}.

\providecommand{\bysame}{\leavevmode\hbox to3em{\hrulefill}\thinspace}

\end{document}